\documentclass[10pt]{article}
\usepackage[T1]{fontenc}
\usepackage[ansinew]{inputenc}
\usepackage{amsthm,amsmath,amssymb}
\usepackage{amsfonts,amscd}
\usepackage{float,mathrsfs,bm,multirow}
\usepackage{graphicx}
\usepackage{wrapfig}
\usepackage{subfig}
\usepackage{enumerate}
\usepackage[percent]{overpic}
\usepackage[all]{xy}

\renewcommand{\S}{S^1}
\newcommand{\R}{\mathbb R}

\newcommand{\llangle}{\langle\!\langle}
\newcommand{\rrangle}{\rangle\!\rangle}
\newcommand{\sect}{\text{\upshape sec}}
\newcommand{\proofbegin}{\noindent{\it Proof.\,\,}}
\newcommand{\proofend}{\hfill$\Box$\bigskip}

\allowdisplaybreaks

\newcommand{\Diff}{\text{\upshape Diff}}
\newcommand{\dS}{\mathcal{S}}
\newcommand{\id}{\text{\upshape id}}
\newcommand{\projK}{p}
\newcommand{\projV}{q}

\DeclareMathOperator{\arctanh}{arctanh}

\newtheorem{thm}{Theorem}[section]
\newtheorem{cor}{Corollary}[section]
\newtheorem{lem}{Lemma}[section]
\newtheorem{prop}{Proposition}[section]

\newtheorem{defn}{Definition}
\newtheorem{rem}{Remark}[section]

\def\theequation{\thesection.\arabic{equation}}
\newcommand{\nequation}{\setcounter{equation}{0}}

\begin{document}

\input epsf
\title{The Hunter-Saxton system and \\ the geodesics on a pseudosphere}

\author{ {\sc Jonatan Lenells}\footnote{Department of Mathematics, Baylor University, One Bear Place $\#$ 97328, Waco, TX 76798, USA} \qquad
{\sc Marcus Wunsch}\footnote{Departement Mathematik, ETH Z\"urich, R\"amistrasse 101, 8092 Zurich, Switzerland \vfill Email: \texttt{Jonatan\_Lenells@baylor.edu, Marcus.Wunsch@math.ethz.ch}}}
 
\date{}

\maketitle 

\begin{abstract}
\noindent 
We show that the two-component Hunter-Saxton system with negative coupling constant describes the geodesic flow on an infinite-dimensional pseudosphere. This approach yields explicit solution formulae for the Hunter-Saxton system. Using this geometric intuition, we conclude by constructing global weak solutions. The main novelty compared with similar previous studies is that the metric is indefinite.
\end{abstract}




\section{Introduction}\nequation
In this paper, we are concerned with the following two-component Hunter-Saxton system subject to periodic boundary conditions: 
\begin{align} \label{2hs}
\begin{cases}
m_t + u m_x + 2 u_x m + \kappa \rho \rho_x = 0, \\ 
\rho_t + (u \rho)_x = 0, \\
 u(0,x) = u_0(x), \; \rho(0,x) = \rho_0(x), 
\end{cases}
 \qquad t > 0,\; x \in \S \simeq \R/\mathbb Z, 
\end{align}
where $m = -\partial_x^2 u$ and $\kappa = \pm 1$ is a parameter, see \cite{kohlmann, len_kahler, ll, liyin1,liyin2,wuwu,Wun09,Wun10,Wun11b}. 
The system (\ref{2hs}) generalizes the well-known Hunter-Saxton equation, $u_{txx} + u u_{xxx} + 2 u_x u_{xx} = 0$, modeling the propagation of nonlinear orientation waves in a massive nematic liquid crystal (cf. \cite{bc, bhr, h-s, km, lenells, tiglay, Y2004}), to which \eqref{2hs} reduces if $\rho_0$ is chosen to vanish identically. 

\par

In mathematical physics, the Hunter-Saxton system \eqref{2hs} is a special instance of the Gurevich-Zybin model describing the nonlinear dynamics of non-dissipative dark matter in one space dimension, as well as a model for nonlinear ion-acoustic waves (see \cite{pavlov1, pavlov} and the references therein). 
Additionally, it is the short wave (or high-frequency) limit of the two-component Camassa-Holm system originating in the Green-Naghdi equations which approximate the governing equations for water waves  \cite{CI2008, ELY07, guo, gz}. The Camassa-Holm system is obtained by setting $m = (1 - \partial_x^2)u$ and $\kappa = 1$ in \eqref{2hs}; the case $\kappa = -1$ corresponds to the situation in which the gravity acceleration points upwards \cite{CI2008}. 
Let us also mention that the Hunter-Saxton system is embedded in a wider class of coupled third-order systems encompassing the axisymmetric Euler flow with swirl \cite{HL2008} and a vorticity model equation \cite{clm, OSW2008} among others (cf. \cite{Wun11b} and the references therein). 

\par

Geometric aspects of \eqref{2hs} have recently been highlighted in \cite{EE}: If $\kappa = 1$, the Hunter-Saxton system can be realized as a geodesic equation of a Riemannian connection on the semi-direct product of a subgroup of the group of orientation-preserving circle diffeomorphisms with the space of smooth functions on the circle. The arguably most prominent geodesic equations are the Euler equations of hydrodynamics \cite{A1966, ebi-mar,arkh} governing the geodesic flow on the Lie group of volume-preserving diffeomorphisms; others include the Camassa-Holm equation \cite{misiolek,kouranbaeva, CK2002, ck:geod}, the Degasperis-Procesi equation \cite{deg-proc,EK09}, their two-component generalizations \cite{ekl10}, and the CLM vorticity model equation \cite{ekw,EW}. 

\par

In \cite{Wun11b}, the second author constructed global weak solutions to the periodic two-component Hunter-Saxton system \eqref{2hs} in the case $\kappa = 1$. These both spatially and temporally periodic solutions are conservative in the sense that the energy $\|u_x(t,.)\|_{L_2(\S)}^2 + \| \rho(t,.) \|_{L_2(\S)}^2$ is constant for almost all times. This construction was given a geometric rationale in \cite{len_kahler}, where it was explained that the possibility of extending geodesics beyond their breaking points is due to an isometry between the underlying space and (a subset of) a unit sphere. 
Using completely different methods, the authors of \cite{liyin2} proved that there are dissipative solutions to the more general $\mu$-Hunter-Saxton system on the real line ($\mu$ here denotes the mean value of the first component $u$), while it was shown in \cite{GY2011} that there are weak solutions of the Hunter-Saxton system on $\R$ when $\kappa = -1$. 

\subsection*{Outline of the paper}
In this paper, we analyze the system (\ref{2hs}) with $\kappa = -1$ in the periodic setting. 

In Section \ref{explicit}, we present explicit solution formulae for (\ref{2hs}) with $\kappa = -1$ using the method of characteristics. It turns out that some solutions exist globally, whereas others develop singularities in finite time.

In Section \ref{geometrysec}, we inquire into the geometry of the Hunter-Saxton system. We show that (\ref{2hs}) with $\kappa = -1$ is the geodesic equation on an infinite-dimensional Lie group $G^s$ equipped with a right-invariant {\it pseudo}-Riemannian metric (i.e. a metric which is nondegenerate, but not positive definite). In fact, we show that $G^s$ is isometric to a subset of an infinite-dimensional pseudosphere. This is the first example known to the authors where a PDE of this type arises as the geodesic equation on a manifold with an indefinite metric.
Since the geodesics on a pseudosphere can easily be written down explicitly, this yields an alternative derivation of the explicit solution formulae of Section \ref{explicit}.  
We also consider the restriction of (\ref{2hs}) to solutions $(u, \rho)$, where $\rho$ has zero mean. Geometrically, this gives rise to the study of a quotient manifold $K^s = G^s /\R$ that admits a natural symplectic structure.
The geometric properties of (\ref{2hs}) can be summarized as follows:
\begin{center}
\begin{tabular}{c|c|c|c}
 & {\bf Type of metric} & {\bf Curvature} & {\bf Underlying geometry} \\
  \hline
 $\kappa = 1$ & Riemannian & constant and positive & spherical \\
	\hline
$\kappa = -1$ & pseudo-Riemannian & constant and positive & pseudospherical 
\end{tabular}
\end{center}

\noindent
In Section \ref{main}, we describe how global weak solutions can be constructed if $\| \rho_0 \|_{L_2(\S)}^2 > \|u_{0x}\|_{L_2(\S)}^2$.

\subsection*{Notation}
The Hilbert space of functions $f : \S \rightarrow \R$ which, together with their derivatives of order $s \ge 0$, are square-integrable, will be denoted by $H^s(\S)$. If $s = 0$, we use the notation $L_2(\S)$ instead of $H^0(\S)$. The subspace of functions $u \in H^s(\S)$ such that $u(0) = 0$ will be denoted by $H_0^s(\S)$. The subspace of functions $u \in H^s(\S)$ such that $\int_{S^1} u dx = 0$ will be denoted by $H_\R^s(\S)$. Lastly, the shorthand $\{ f > 0\}$ for the set $\{ x \in \S:\; f(x) > 0 \}$ (and other analogous short forms) will be used throughout the text. 

\pagebreak

\section{Explicit solution formulae} \label{explicit} \nequation
In this section, we provide new solution formulae for the Hunter-Saxton system (\ref{2hs}). 
\\ 
Let $\kappa = -1$. The first component equation of (\ref{2hs}) can be rewritten in terms of the gradient $u_x(t,x)$ as
\begin{equation} 
u_{tx} + u u_{xx} + \frac{1}{2} u_x^2 + \frac{1}{2} \rho^2 + 2c = 0,  
\end{equation} 
\noindent
where the nonlocal term $c = \frac{1}{4} \int_{\S} \left(u_x(t,x)^2 - \rho(t,x)^2\right)\;dx$ is enforced by periodicity. It follows from (\ref{2hs}) that $c$ is indepedent of $t$ (see \cite{wuwu, Wun09}). Three possibilities now arise: 
\begin{enumerate}[$(i)$]
\item $c > 0$; 
\item $c = 0$; 
\item $c < 0$. 
\end{enumerate}
Since the Hunter-Saxton system is invariant under the scalings 
$$u(t,x) \mapsto \alpha u(\alpha t,x), \quad \rho(t,x) \mapsto \alpha \rho(\alpha t,x), \qquad \alpha \in \R,$$ 
we may without loss of generality set $c = 1$ for case $(i)$, and $c = -1$ for case $(iii)$. 

Let us introduce the Lagrangian flow map $\varphi$ solving 
\begin{equation} \label{flow}
\varphi_t(t,x) = u(t,\varphi(t,x)),\quad \varphi(0,x) = x \in \S.
\end{equation}
\noindent
In terms of $U(t,x) := u_x(t,\varphi(t,x))$ and $\varrho(t,x) := \rho(t,\varphi(t,x))$, we can rephrase  \eqref{2hs} as
\begin{equation} \label{lagran}
\begin{cases}
U_t + \frac{1}{2} U^2 + \frac{1}{2} \varrho^2 + 2c = 0,	\\
\varrho_t + U \varrho = 0, \\
U(0,x) =u_{0x}(x),\; \varrho(0,x) = \rho_0(x). 
\end{cases}
\end{equation}
The sum $p := U + \varrho$ and the difference $q := U - \varrho$ both satisfy the following Riccati-type differential equation with corresponding initial data:  
\begin{equation}
z_t = - \frac{1}{2} z^2 - 2c, \quad z(0,x) = z_0(x) = 
\begin{cases} \label{ric}
p(0,x) = u_{0x}(x) + \rho_0(x); \\
q(0,x) = u_{0x}(x) - \rho_0(x).
\end{cases}
\end{equation}
Equation \eqref{ric} is explicitly solvable: 
\begin{align}
z(t,x) = 
\begin{cases} 
\dfrac{2 z_0(x) - 4\tan t}{z_0(x) \tan t + 2}  \qquad &\mbox{if } c = 1; 
\\[0.3cm]
\dfrac{2z_0(x)}{2 + z_0(x)\;t} \qquad &\mbox{if } c = 0; 
\\[0.3cm] 
2 \; \dfrac{z_0(x) - 2 + e^{2t}\;(2 + z_0(x))}{2 - z_0(x) + e^{2t}\; (2 + z_0(x))} \qquad &\mbox{if } c = -1. 
\end{cases}
\end{align}
Decomposition yields for the first component 
\begin{align*}
U(t,x) = 
\begin{cases}
\dfrac{4\;\cos(2t) \;u_{0x}(x) + \sin(2t)\; \left[ u_{0x}(x)^2 - \rho_0(x)^2 - 4 \right]}
{\left[ u_{0x}(x)\; \sin{t} + 2\;\cos{t}\right]^2 - \rho_0(x)^2\;\sin^2{t}} \quad &\mbox{if } c = 1; 
\\[0.4cm] 
\dfrac{4\; u_{0x}(x) + 2\; \left[ u_{0x}(x)^2 - \rho_0(x)^2 \right]\;t}{\left[2 + u_{0x}(x) \; t \right]^2 - \rho_0(x)^2\;t^2} 
&\mbox{if } c = 0; 
\\[0.4cm] 
\dfrac{4\; \cosh(2t)\;u_{0x}(x) + \sinh(2t)\;\left[ u_{0x}(x)^2 - \rho_0(x)^2 + 4 \right]}
{\left[2\; \cosh{t} + u_{0x}\;\sinh{t}\right]^2 - \rho_0(x)^2\;\sinh^2{t}} &\mbox{if } c = -1; 
\end{cases}
\end{align*}
and, for the second, 
\begin{align*}
\varrho(t,x) = 
\begin{cases}
\dfrac{4\rho_0(x)}{\left[ u_{0x}(x)\; \sin{t} + 2\;\cos{t}\right]^2 - \rho_0(x)^2\;\sin^2{t}} \quad &\mbox{if } c = 1; 
\\[0.4cm] 
\dfrac{4\;\rho_0(x)}{\left[ 2 + u_{0x}(x)\;t\right]^2 - \rho_0(x)^2\;t^2} &\mbox{if } c = 0; 
\\[0.4cm] 
\dfrac{4\;\rho_0(x)}{\left[2\; \cosh{t} + u_{0x}\;\sinh{t}\right]^2 - \rho_0(x)^2\;\sinh^2{t}} &\mbox{if } c = -1.
\end{cases}
\end{align*} 
These solutions do not exist beyond a critical time $T^* > 0$ given by 
\begin{align} \label{singularity}
T^* = 
\begin{cases} 
\dfrac{\pi}{2} + \min\left\{ \arctan\left[ \min_{\S} \dfrac{u_{0x} - \rho_0}{2} \right], \arctan\left[  \min_{\S} \dfrac{u_{0x} + \rho_0}{2} \right] \right\} \; &\mbox{if } c = 1; 
\\[0.4cm]
\min\left\{ \inf_{\{ u_{0x} < \rho_0 \}} \left\{ \frac{-2}{u_{0x} - \rho_0} \right\}, \inf_{\{ u_{0x} + \rho_0 < 0 \}} \left\{ \frac{-2}{u_{0x} + \rho_0} \right\}\right\} &\mbox{if } c = 0; 
\\[0.4cm] 
\min \left\{ {\rm arccoth}\left[\inf_{\{ \rho_0 - u_{0x} > 2\}} \dfrac{\rho_0 -u_{0x}}{2}\right], 
{\rm arccoth} \left[ \inf_{\{u_{0x} +\rho_0 < - 2 \}} \dfrac{-u_{0x} - \rho_0}{2} \right] \right\}
&\mbox{if } c = -1. 
\end{cases}
\end{align}

\noindent
We have thus proven the following proposition. 
\begin{prop} \label{expli}
Let $s \geq 2$. Suppose $(u_0, \rho_0) \in H_0^s(\S) \times H^{s - 1}(\S)$, and denote by  
$$
(u, \rho) \in C([0,T^*); H_0^s(\S) \times H^{s - 1}(\S)) \cap C^1([0,T^*); H_0^{s - 1}(\S) \times H^{s - 2}(\S))
$$ 
the solution of (\ref{2hs}) with $\kappa = -1$ with initial data $(u_0,\rho_0)$. This solution exists and is unique, see \cite{wuwu}. 
Furthermore, let $\varphi(t,x)$ with $\varphi(t,0) = 0$, $t \in [0,T^*)$, solve the Lagrangian flow map equation \eqref{flow}. Then
\begin{align} \label{phix}
\varphi (t,x) = \begin{cases}
\int_0^x \Bigl[\left( \cos t + \frac{u_{0x}(y)}{2} \; \sin t\right)^2 - \frac{\rho_0(y)^2}{4}\; \sin^2 t \Bigr]\;dy, & c = 1,
	\\[0.4cm] 
\int_0^x \Bigl[\left( 1 + \frac{u_{0x}(y)}{2}\;t \right)^2 - \frac{\rho_0(y)^2}{4}\;t^2 \Bigr]\;dy, & c = 0,
 	\\[0.4cm] 
 \int_0^x \Bigl[\left( \cosh t + \frac{u_{0x}(y)}{2} \; \sinh t\right)^2 - \frac{\rho_0(y)^2}{4}\; \sinh^2 t \Bigr]\;dy, & c = -1. 
 \end{cases}
 \end{align}
The first time when $\varphi(t,.)$ ceases to be injective is given in \eqref{singularity}. 
In the case of $c = -1$, the solution exists indefinitely if and only if 
\begin{equation} \label{globex}
|\rho_0(x)| \le u_{0x}(x) + 2 \qquad \text{for all} \quad  x \in \S,
\end{equation}
while all solutions when $c = 1$ and $c = 0$ inevitably develop singularities in finite time. 
\end{prop}

\begin{rem} In the case of $c = 1$, the above solution formula was found in \cite{kohlmann} using a different approach. 
\end{rem}

\begin{rem}
If $\kappa = 1$, then the solution to the Lagrangian flow map equation \eqref{flow} takes the form (cf. \cite{Wun11b, len_kahler})
$$
\varphi (t,x) = \int_0^x \left\{ \left( \cos t + \frac{u_{0x}(y)}{2} \; \sin t\right)^2 + \frac{\rho_0(y)^2}{4}\; \sin^2 t \right\}\;dy, 
$$
where we have assumed the normalization $c = \frac{1}{4} \int_{\S} \left(u_{0x}^2 +  \kappa \rho_0^2 \right) dx = 1$. 
\end{rem}

\begin{figure}
  \centering
  \subfloat[$u_{0x} = \cos(2\pi x)$, $\rho_0 = 3 \cos(2\pi x)$]{\includegraphics[width=0.32\textwidth]{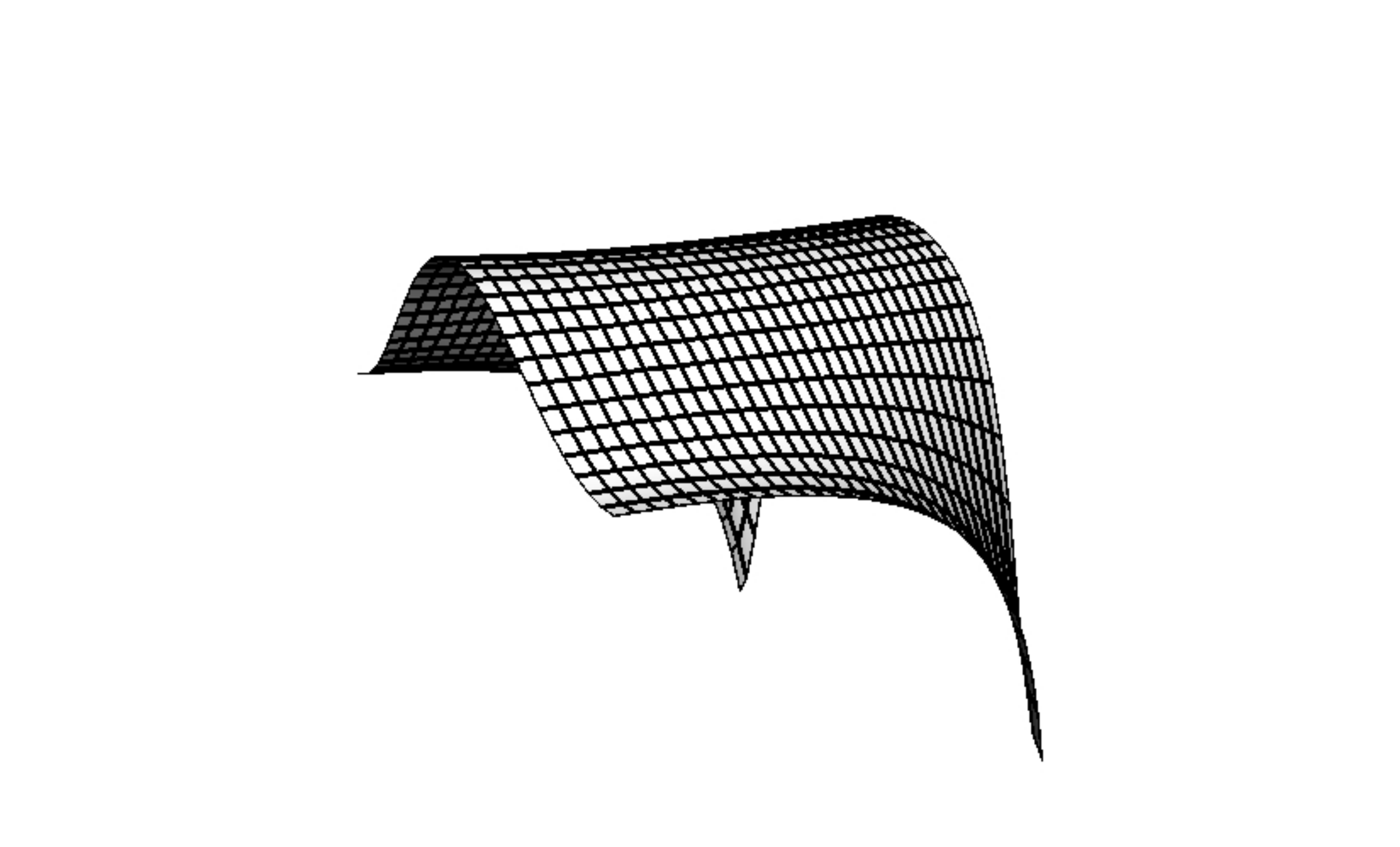}}                
  \subfloat[$u_{0x} = \cos(2\pi x)$, $\rho_0 = \frac{3}{\sqrt{2}}$]{\includegraphics[width=0.32\textwidth]{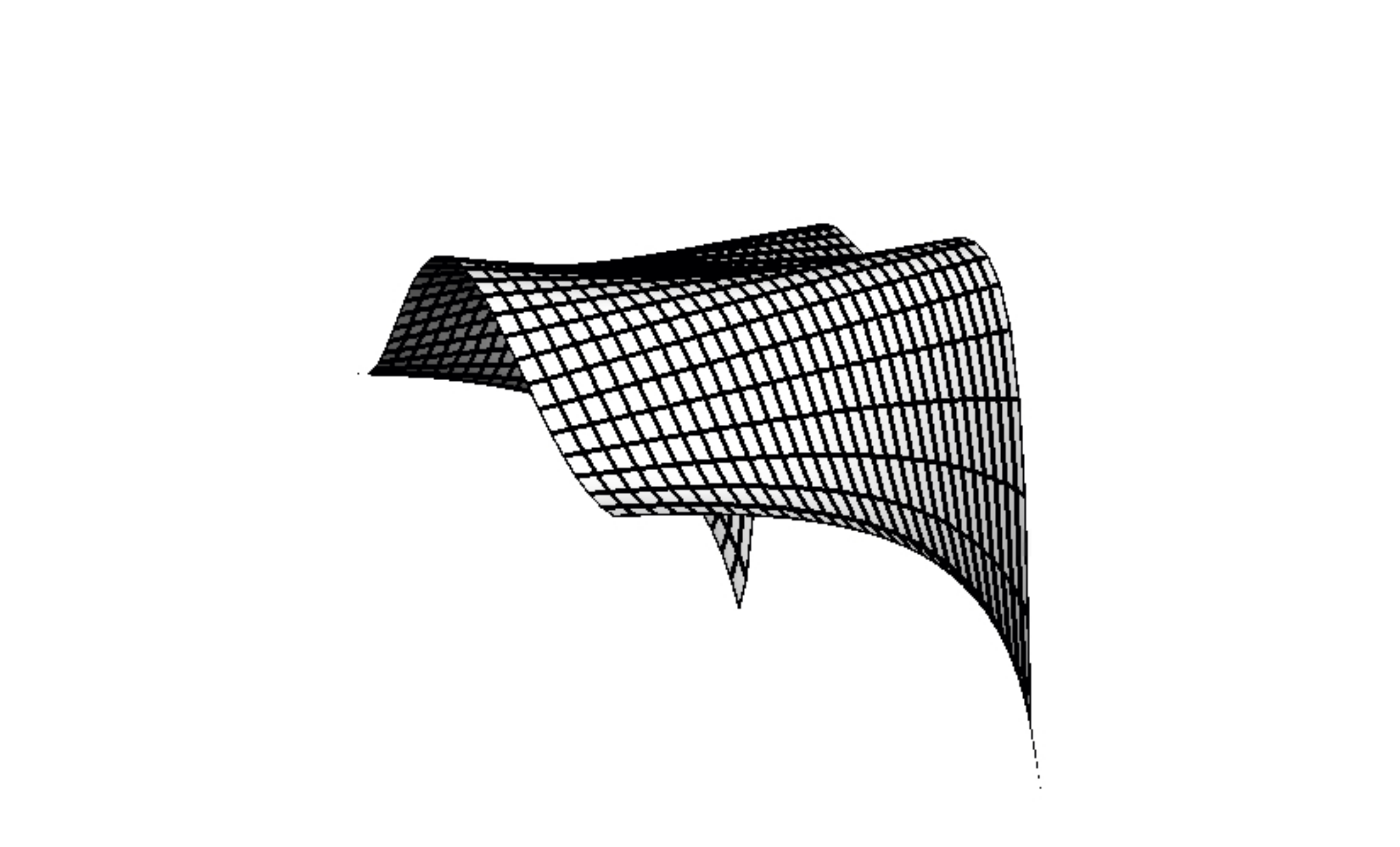}}
  \subfloat[$u_{0x} = \cos(2\pi x)$, $\rho_0 = \cos(2\pi x) + 2$]{\includegraphics[width=0.32\textwidth]{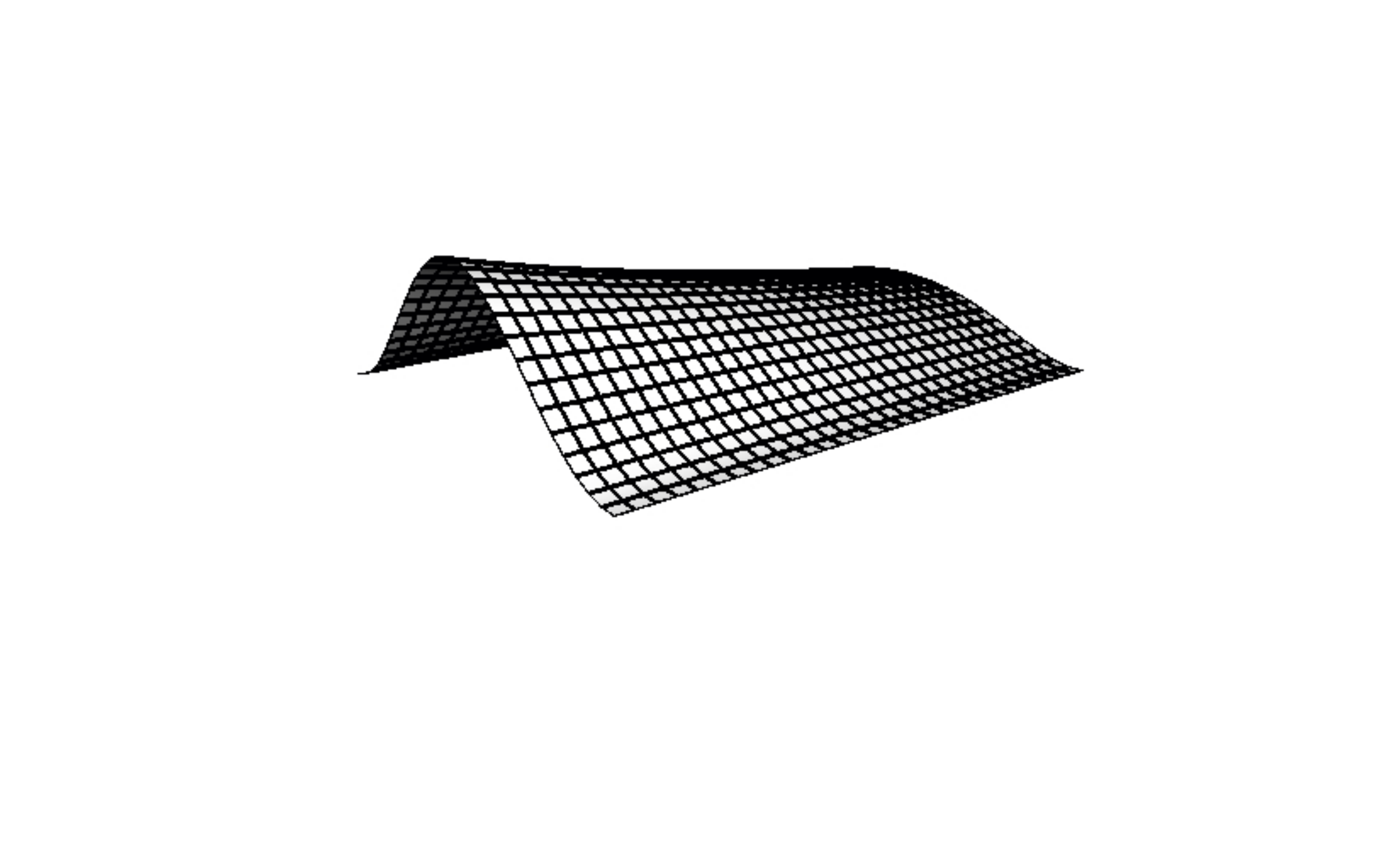}}
  \caption{{\sc Plots of $u_x(t,\varphi(t,x))$ corresponding to three different sets of initial data with 
  $c = -1$.} Note that the initial data in (a) and (b) violate the inequality \eqref{globex}, so that the emanating solutions break down in finite time, while the initial data set of (c) fulfills this condition and thus gives rise to global solutions.}
  \label{fig:a=2}
\end{figure}

\section{The geometry of the Hunter-Saxton system} \label{geometrysec}\nequation
Equation (\ref{2hs}) with $\kappa = 1$ describes the geodesic flow on a sphere \cite{len_kahler}. Here, we will show that \eqref{2hs} with $\kappa = -1$ describes the geodesic flow on a pseudosphere. More precisely, we will show that (\ref{2hs}) with $\kappa = -1$ is the Euler equation for the geodesic flow on a pseudo-Riemannian manifold $G^s$ and that $G^s$ is isomorphic to a subset of the unit pseudosphere in $L^2(S^1; \R^2)$.

\subsection{Preliminaries}
Suppose $s > 5/2$. Let $\Diff^s(S^1)$ denote the Banach manifold of orientation-preserving diffeomorphisms of $S^1$ of Sobolev class $H^s$. Let $\Diff_0^s(S^1)$ denote the subgroup of $\Diff^s(S^1)$ consisting of diffeomorphisms $\varphi$ such that $\varphi(0) = 0$.
Let $G^s$ denote the semidirect product $\Diff_0^s(\S) \circledS H^{s-1}(\S)$ with multiplication given by 
$$(\varphi, \alpha)(\psi, \beta) = (\varphi\circ\psi, \beta + \alpha \circ \psi).$$
The nondegenerate metric $\langle \cdot, \cdot \rangle$ on $G^s$ is defined at the identity by 
\begin{align}\label{metricatid}  
  \langle (u_1, u_2), (v_1, v_2) \rangle_{(\id, 0)} = \frac{1}{4}\int_{\S}\left(u_{1x} v_{1x} + \kappa u_2 v_2 \right) dx,
\end{align}
and extended to all of $G^s$ by right invariance, i.e.
\begin{align}\label{metric}
\langle U, V \rangle_{(\varphi,\alpha)} & = \left\langle (U_1 \circ\varphi^{-1}, U_2 \circ\varphi^{-1}), (V_1 \circ\varphi^{-1}, V_2 \circ\varphi^{-1}) \right\rangle_{(\id,0)}
	\\ \nonumber
& = \frac{1}{4} \int_{\S} \left(\frac{U_{1x}V_{1x}}{\varphi_x} + \kappa U_2V_2 \varphi_x \right) dx,
\end{align}
where $U = (U_1, U_2)$ and $V = (V_1, V_2)$ are elements of $T_{(\varphi, \alpha)}G^s \simeq H_0^{s}(\S) \times H^{s-1}(\S)$. 

Let $A = -\partial_x^2$. Then $A$ is an isomorphism $H_0^s(S^1) \to H_{\R}^{s-2}(S^1)$ 
with inverse given by
\begin{align}\label{Ainv}
  (A^{-1}f)(x) = -\int_0^x\int_0^y f(z)dzdy +  x \int_{S^1}\int_0^y f(z)dz dy
\end{align}
whenever $\int_{S^1} f dx = 0$. 
The following proposition expresses the fact that equation (\ref{2hs}) is the geodesic equation on $G^s$ in the sense that a curve $(\varphi(t), \alpha(t))$ in $G^s$ is a geodesic if and only if $(u(t), \rho(t)) \in T_{(\id, 0)}G^s$ defined by
\begin{align}\label{urhovarphitft}
 (u,\rho) 
 = (\varphi_t\circ\varphi^{-1},\alpha_t\circ\varphi^{-1})
\end{align}
satisfies (\ref{2hs}). 

\begin{prop}\label{geodesicprop}
Let $s > 5/2$. Let $(\varphi, \alpha):J \to G^s$ be a $C^2$-curve where $J \subset \R$ is an open interval and define $(u, \rho)$ by (\ref{urhovarphitft}). Then 
\begin{align}\label{urhoCC1J}
(u, \rho) \in C\left(J; H_0^s(S^1) \times H^{s-1}(S^1)\right) \cap C^1\left(J; H_0^{s-1}(S^1) \times H^{s-2}(S^1)\right)
\end{align}
and $(\varphi, \alpha)$ is a geodesic on $J$ if and only if $(u, \rho)$ satisfies the following weak form of (\ref{2hs}) for $t \in J$:
\begin{align}\label{weak2HS}
  & \begin{pmatrix} u_t + uu_x \\ \rho_t + u\rho_x \end{pmatrix} 
  = \begin{pmatrix} 
  - \frac{1}{2} A^{-1}\partial_x\bigl(u_x^2 + \kappa \rho^2\bigr) \\ -\rho u_x
  \end{pmatrix}.
\end{align}
\end{prop}
\proofbegin
The case of $\kappa = 1$ was treated in Proposition 4.1 of \cite{len_kahler}; the proof when $\kappa = -1$ is similar.
\proofend

\subsection{A pseudosphere}
Let $\dS$ denote the unit pseudosphere in $L_2(\S; \R^2)$ defined by
$$\dS = \biggl\{(f_1, f_2) \in L_2(\S; \R^2) \; \bigg| \; \int_{\S} \left(f_1^2(x) - f_2^2(x)\right) dx = 1\biggr\}.$$
Let $\dS^s$ denote the elements in $\dS$ that are of Sobolev class $H^s$.
Then $\dS^s$ is a Banach submanifold of $H^s(\S; \R^2)$ (cf. \cite{L1995} p. 29).
The indefinite scalar product on $L_2(\S; \R^2)$ defined for $X = (X_1, X_2)$ and $Y = (Y_1, Y_2)$ in $L_2(S^1; \R^2)$ by
$$\llangle X, Y \rrangle = \int_{\S} (X_1Y_1 - X_2Y_2) dx,$$
induces a weak pseudo-Riemannian metric $\llangle \cdot, \cdot \rrangle$ on $\dS^{s}$.

\begin{rem}
Recall that the $n$-dimensional pseudosphere $S_\nu^{n}(r)$ of index $\nu$ and radius $r>0$ is defined as the submanifold 
\begin{align}\label{pseudospheredef}
  S^n_\nu(r) = \biggl\{(x_1, \dots, x_{n+1}) \in \R^{n+1} \bigg|  - \sum_{i=1}^{\nu} x_{i}^2 + \sum_{i=\nu+1}^{n+1} x_i^2 = r^2\biggr\},
\end{align}  
equipped with the pseudo-Riemannian metric induced by the indefinite bilinear form 
$$ds^2 = - \sum_{i=1}^{\nu} dx_{i}^2 + \sum_{i=\nu+1}^{n+1} dx_i^2.$$
If $\nu = 1$, pseudospheres are the Minkowskian analogs of spheres in Euclidean space. The curvature of $S_\nu^n(r)$ is constant and equal to $1/r^2$. 
We refer to \cite{oneill, W2011} for more background on (finite-dimensional) pseudospheres. 
\end{rem}

\noindent
We let $\mathcal{U}^{s} \subset L^2(S^1; \R^2)$ denote the following open subset of $\dS^{s}$:
\begin{align}\label{Usdef}
\mathcal{U}^{s} = \left\{(f_1, f_2) \in \dS^s \;\middle|\; f_1(x) > 0 \; \hbox{and} \; f_1^2(x) - f_2^2(x) > 0 \; \hbox{for all} \; x \in S^1\right\},
\end{align}
and equip $\mathcal{U}^{s}$ with the manifold structure and metric inherited from $\dS^{s}$.

\begin{thm}\label{sphereth}
The space $(G^s, \langle \cdot, \cdot \rangle)$ is isometric to a subset of the unit pseudosphere in $L^2(S^1; \R^2)$. 
More precisely, for any $s > 5/2$, the map $\Phi:G^s \to \mathcal{U}^{s-1} \subset \dS^{s-1}$ defined by
$$\Phi(\varphi, \alpha) =  \sqrt{\varphi_x}(\cosh(\alpha/2), \sinh(\alpha/2))$$
is a diffeomorphism and an isometry.
\end{thm}
\proofbegin
If $f = (f_1, f_2) \in \mathcal{U}^{s-1}$, then the function $\varphi(x) = \int_0^x (f_1^2(y) - f_2^2(y)) dy$ satisfies
$\varphi(0) = 0$, $\varphi(1) = 1$, $\varphi_x = f_1^2 - f_2^2 \in H^{s-1}(S^1)$, and $\varphi_x > 0$, while the function $\alpha(x) = 2\arctanh{\frac{f_2(x)}{f_1(x)}}$ belongs to $H^{s-1}(S^1)$.
Thus, using the identities
$$\cosh(\arctanh{x}) = \frac{1}{\sqrt{1 - x^2}}, \qquad
\sinh(\arctanh{x}) =  \frac{x}{\sqrt{1 - x^2}}, \qquad -1 < x < 1,$$
we find that the inverse of $\Phi$ is given explicitly by
\begin{align}\label{Phiinverse}
  \Phi^{-1}(f) = \left(\int_0^x (f_1^2(y) - f_2^2(y)) dy, 2\arctanh{\frac{f_2(x)}{f_1(x)}}\right), \qquad f \in \mathcal{U}^{s-1}.
\end{align}  
This shows that $\Phi$ is bijective. Since both $\Phi$ and $\Phi^{-1}$ are smooth, $\Phi$ is a diffeomorphism.

Using that
$$T_{(\varphi, \alpha)}\Phi (U_1, U_2)
= \frac{1}{2\sqrt{\varphi_x}}\left(U_{1x}\cosh{\frac{\alpha}{2}} + \varphi_xU_2 \sinh{\frac{\alpha}{2}}, U_{1x}\sinh{\frac{\alpha}{2}} + \varphi_xU_2 \cosh{\frac{\alpha}{2}}\right),$$
we find that
\begin{align*}
 \llangle T_{(\varphi, \alpha)} (&U_1, U_2),  T_{(\varphi, \alpha)} (V_1, V_2)\rrangle
	\\
& = \int_{\S} \frac{1}{4\varphi_x} \biggl\{\left(U_{1x}\cosh{\frac{\alpha}{2}} + \varphi_xU_2 \sinh{\frac{\alpha}{2}}\right)
\left(V_{1x}\cosh{\frac{\alpha}{2}} + \varphi_xV_2 \sinh{\frac{\alpha}{2}}\right)
	\\
& \hspace{2cm} - \left(U_{1x}\sinh{\frac{\alpha}{2}} + \varphi_xU_2 \cosh{\frac{\alpha}{2}}\right)
\left(V_{1x}\sinh{\frac{\alpha}{2}} + \varphi_xV_2 \cosh{\frac{\alpha}{2}}\right)\biggr\} dx
	\\
& = \int_{\S} \frac{1}{4\varphi_x} \biggl\{ U_{1x} V_{1x} - \varphi_x^2 U_2 V_2\biggr\} dx
= \langle (U_1, U_2), (V_1, V_2) \rangle_{(\varphi, \alpha)},
\end{align*}
whenever $(U_1, U_2)$ and $(V_1, V_2)$ belong to $T_{(\varphi, \alpha)}G^s$.
This shows that $\Phi$ is an isometry. 
\proofend

\begin{cor}\label{curvcor}
The sectional curvature of $(G^s, \langle \cdot, \cdot \rangle)$ is constant and equal to $1$.
\end{cor}
\proofbegin 
In view of Theorem \ref{sphereth}, it is enough to prove that the unit pseudosphere $\dS^{s}$ has constant sectional curvature equal to $1$. As in the finite-dimensional case, this can be proved using the Gauss equation.\footnote{The Gauss equation holds also for pseudo-Riemannian Banach manifolds, cf. \cite{oneill} p. 100 and \cite{L1995} p. 390.}
Indeed, let $n$ denote the outward normal to $\dS^s \subset H^s(\S; \R^2)$. Since the outward normal to $\dS^s$ at $f$ is $f$ itself, $n$ is the identity map.  
Moreover, the tangent space at a point $f = (f_1, f_2) \in \dS^s$ is given by
$$T_f \dS^s = \left\{X=(X_1, X_2) \in H^s(\S; \R^2) \; \middle| \; \int_{\S} (f_1 X_1 - f_2 X_2) dx = 0\right\},$$
and the metric connection on $\dS$ is given by (see \cite{oneill} p. 99) 
\begin{align}\label{sphereconnection}
  (\nabla_X Y)(f) = (DY \cdot X)^T
\end{align}
where $Z^T$ denotes the orthogonal projection of a vector $Z \in T_f H^s(\S; \R^2)$ onto $T_f \dS^s$ with respect to $\llangle \cdot, \cdot \rrangle$.
Thus, if $X$ is a vector field on $\dS^s$,
$$\nabla_X n = (Dn \cdot X)^T = X.$$
It follows that the second fundamental form $\Pi$ is given by
$$\Pi(X, Y) = - \llangle \nabla_X n, Y \rrangle n = - \llangle X, Y \rrangle n,$$
where $X,Y$ are vector fields on $\dS^s$. Consequently, if $X$ and $Y$ are orthonormal, the curvature tensor $R$ on $\dS^s$ satisfies
\begin{align*}
\llangle R(X,Y)Y, X \rrangle
& = \llangle \Pi(X,X), \Pi(Y,Y) \rrangle
- \llangle \Pi(X,Y), \Pi(Y,X) \rrangle
	\\
& = \llangle X,X \rrangle \llangle Y,Y \rrangle
- \llangle X, Y \rrangle \llangle Y,X \rrangle  = 1.	
\end{align*}
\proofend

\begin{rem}
In Appendix \ref{Aapp}, we give an alternative direct proof of Corollary \ref{curvcor} which does not rely on the isometry of Theorem \ref{sphereth}.
\end{rem}

\subsection{Geodesics}
Just like in the finite-dimensional case, we can write down explicit formulas for the geodesics on the pseudosphere $\mathcal{S}^s$. In this way, we recover the explicit solution formulas of Section \ref{explicit} for the Hunter-Saxton system.
Moreover, these geodesics are naturally divided into three types---spacelike, lightlike, and timelike---and these types correspond to the three cases distinguished in Section \ref{explicit}.

\begin{thm}\label{geoth}
Let $s > 5/2$ and let $(u_0, \rho_0) \in H_0^s(\S) \times H^{s-1}(\S)$.
Let $f(t) = (f_1(t), f_2(t))$ denote the unique geodesic in $\dS^{s-1}$ such that
\begin{align}\label{geodesicinitialconditions}
  f(0) = (1,0), \qquad f_t(0) = \frac{1}{2}(u_{0x}, \rho_0).
\end{align}  
Define $c \in \R$ by 
$$c = \llangle f_t(0), f_t(0) \rrangle = \frac{1}{4}\int_{\S} (u_{0x}^2 - \rho_0^2)dx.$$ 
Then
\begin{align}\label{fgeodesicformula}
  f(t) = \begin{cases} 
  \left(\cos(\sqrt{c}t) + \frac{u_{0x}}{2\sqrt{c}}\sin(\sqrt{c}t), \frac{\rho_0}{2\sqrt{c}}\sin(\sqrt{c}t)\right), & c > 0 \; \text{(spacelike)}, \\
 \left(\frac{u_{0x}}{2}t + 1, \frac{\rho_0}{2} t \right), & c = 0 \;\text{(lightlike)}, \\
 \left(\cosh(\sqrt{|c|}t) + \frac{u_{0x}}{2\sqrt{|c|}}\sinh(\sqrt{|c|}t), \frac{\rho_0}{2\sqrt{|c|}}\sinh(\sqrt{|c|}t)\right), & c < 0 \;\text{(timelike)}.
 \end{cases}
\end{align}
\end{thm} 

\begin{figure} 
  \centering
    \includegraphics[width=0.48\textwidth]{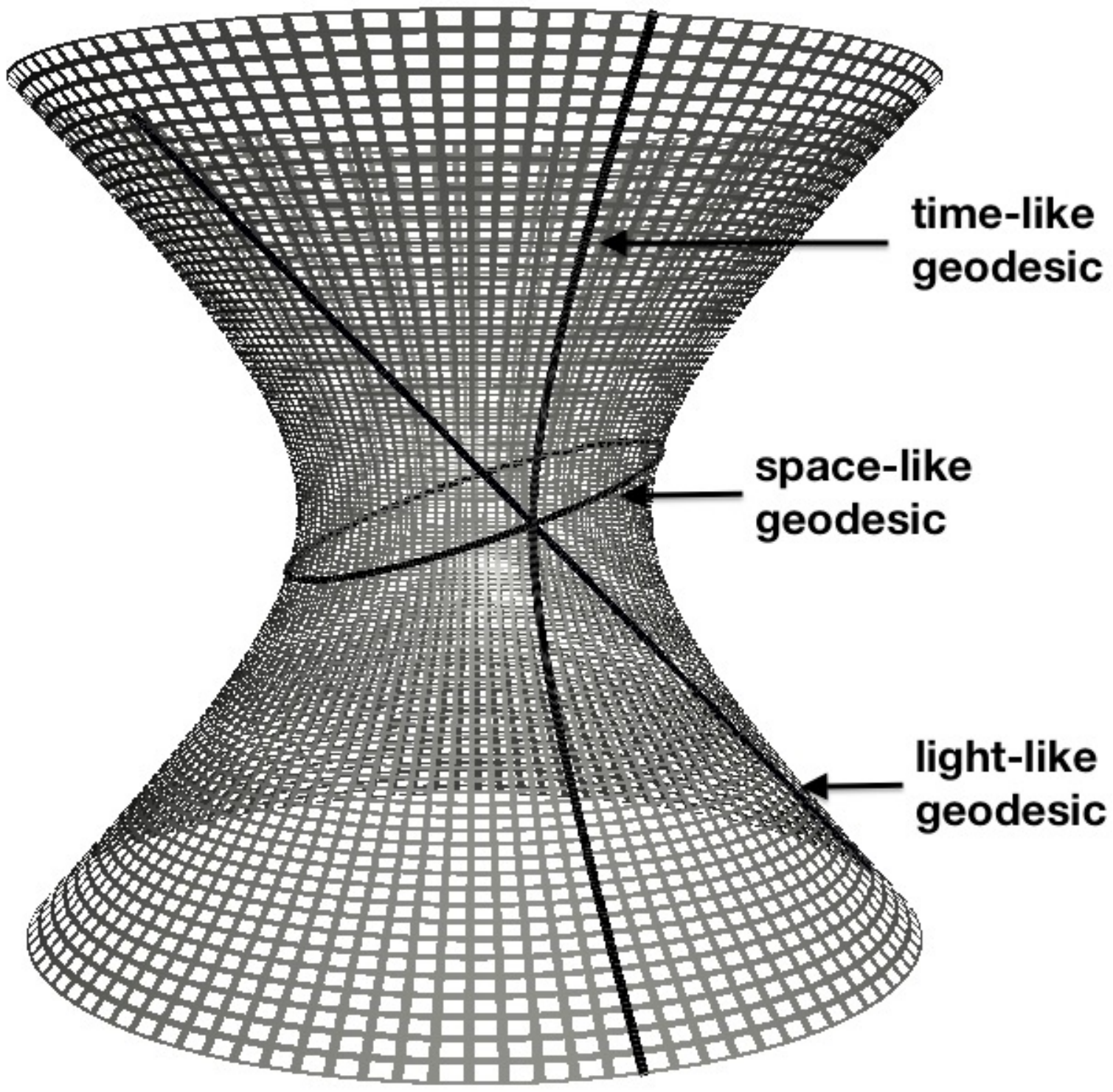}
  \caption{{\sc A pseudosphere with three different kinds of geodesics.}}
  \label{psi_sphere}
\end{figure}

\proofbegin
A straightforward computation shows that $f(t)$ satisfies the initial conditions (\ref{geodesicinitialconditions}) for any $c \in \R$. We can also check that
$$\int_{\S} (f_1(t)^2 - f_2(t)^2) dx = 1$$
for all $t$, showing that $f(t)$ is a curve in $\dS^{s-1}$.
Moreover, for any $c \in \R$, $f$ satisfies the equation
$$f_{tt} + c f = 0.$$ 
Since the tangential and normal parts of $Z = Z^T + Z^N \in T_f H^s(\S; \R^2)$ are given by 
\begin{align*}
  Z^T = Z - \frac{\llangle Z, f\rrangle}{\llangle f, f \rrangle} f, \qquad Z^N = \frac{\llangle Z, f\rrangle}{\llangle f, f \rrangle} f,
\end{align*}
the expression (\ref{sphereconnection}) for the covariant derivative yields
$$\nabla_{f_t} f_t = (f_{tt})^T = -cf^T = 0.$$
This proves that $f(t)$ indeed is the correct geodesic.
\proofend

\begin{rem}
1. The three cases $c > 0$, $c = 0$, $c< 0$ correspond to spacelike, lightlike, and timelike geodesics respectively, see {\rm Figure \ref{psi_sphere}}.

2. Theorems \ref{sphereth} and \ref{geoth} imply that the geodesic $(\varphi(t), \alpha(t))$ in $G^s$ starting at $(\id, 0)$ and with initial velocity $(\varphi_t(0), \alpha_t(0)) = (u_0, \rho_0) \in T_{(\id,0)}G^s$ is given by
$$(\varphi(t), \alpha(t)) = \Phi^{-1}(f(t))$$
where $f(t)$ is given by (\ref{fgeodesicformula}). Since $\Phi^{-1}$ is given by (\ref{Phiinverse}), this immediately yields the explicit formulas in (\ref{phix}).

3. It follows from Theorem \ref{geoth} that all solutions of the Hunter-Saxton system -- except those whose initial data satisfy the condition \eqref{globex} -- break down in finite time. The maximal existence time $T^*$ is the first time for which $f(t)$ hits the boundary of $\mathcal{U}^{s-1}$, i.e., $T^*$ is the first time for which $f_1^2(t,x) - f_2^2(t,x) = 0$ for some $x \in \S$.
\end{rem}

\subsection{A symplectic manifold}
The mean value $\int_{S^1} \rho dx$ of the second component $\rho$ of a solution $(u, \rho)$ of (\ref{2hs}) is a conserved quantity.
Thus, if $\rho$ has zero mean initially, it will have zero mean at all later times. This suggests that we consider the following variation of (\ref{2hs}):
\begin{align}\label{pi2HS}
\begin{cases}
  m_t + u m_x + 2 u_x m + \kappa \pi(\rho) \rho_x = 0, \\ 
    \pi(\rho)_t + (\pi(\rho) u)_x = 0,
\end{cases} \qquad t>0, \  x\in S^1,
\end{align}
where $\pi(\rho) =  \rho - \int_{S^1} \rho dx$ denotes the orthogonal projection onto the subspace of functions in $L_2(\S)$ of zero mean. 
For solutions such that $\int_{S^1} \rho dx = 0$, (\ref{pi2HS}) coincides with (\ref{2hs}).
The system (\ref{pi2HS}) with $\kappa = 1$ was analyzed in \cite{len_kahler}; here we will consider the case of $\kappa = -1$. We will see that the system (\ref{pi2HS}) possesses some interesting geometric properties not shared by (\ref{2hs}). 

Let $H^{s}(S^1)/\R$ denote the space $H^{s}(S^1)$ with two functions being identified iff they differ by a constant; the equivalence class of $\alpha \in H^{s}(S^1)$ will be denoted by $[\alpha] \in H^{s}(S^1)/\R$.
We define $K^s$ as the semidirect product $\Diff_0^s(S^1) \circledS (H^{s-1}(S^1)/\R)$ with multiplication given by 
$$(\varphi, [\alpha])(\psi, [\beta]) = (\varphi\circ\psi, [\beta + \alpha \circ \psi]).$$
We equip $K^s$ with the right-invariant metric given at the identity by
\begin{align}\label{pimetricatid}  
  \langle (u, [\rho]), (v, [\tau]) \rangle_{(\id, [0])} = \frac{1}{4}\int_{S^1} (u_x v_x - \pi(\rho)\pi(\tau)) dx.
\end{align}
Extending the projection $\pi$ to any tangent space by right invariance so that
$$\pi(U_2) = U_2 - \int_{S^1} U_2 \varphi_x dx$$
whenever $(U_1, [U_2]) \in T_{(\varphi, \alpha)}K^s \simeq H_0^s(S^1) \times (H^{s-1}(S^1)/\R)$, we find
\begin{align}\label{pimetric}
\langle (U_1, [U_2]), (V_1, [V_2]) \rangle_{(\varphi, [\alpha])} = \frac{1}{4}\int_{S^1} \left(\frac{U_{1x}V_{1x}}{\varphi_x} - \pi(U_2)\pi(V_2)\varphi_x\right) dx.
\end{align}

We define a connection $\nabla$ on $K^s$ by
$$\nabla_X Y = DY \cdot X - \Gamma_{(\varphi, [\alpha])}(Y, X),$$
where the Christoffel map $\Gamma$ is defined for $u = (u_1, [u_2])$, $v = (v_1, [v_2])$ in $T_{(\id, [0])}K^s$ by
\begin{align*}
\Gamma_{(\id, [0])}(u, v) = - \frac{1}{2}\begin{pmatrix} A^{-1}\partial_x(u_{1x}v_{1x} - \pi(u_2)\pi(v_2)) \\
[u_{1x}\pi(v_2) + v_{1x}\pi(u_2)] \end{pmatrix}
\end{align*}
and extended to the tangent space at $(\varphi, [\alpha]) \in K^s$ by right invariance:
\begin{align*}
\Gamma_{(\varphi, [\alpha])}(u\circ \varphi, v \circ \varphi) = \Gamma_{(\id, [0])}(u, v) \circ \varphi.
\end{align*}
We also define a (1,1)-tensor $J$ and a two-form $\omega$ on $K^s$ by
\begin{align}\label{Jdef}
J_{(\varphi, [\alpha])}(U_1, [U_2]) = \left(-\int_0^x \pi(U_2)\varphi_x dy, -\left[\frac{U_{1x}}{\varphi_x}\right]\right)
\end{align}
and
\begin{align}\label{omegadef}
\omega_{(\varphi, [\alpha])}((U_1, [U_2]), (V_1, [V_2]))
= \frac{1}{4}\int_{S^1}(U_{2x}V_{1} - V_{2x}U_{1}) dx
\end{align}
whenever $(U_1, [U_2]), (V_1, [V_2]) \in T_{(\varphi, [\alpha])}K^s$. 
We note that $\omega$ and $J$ are right-invariant. 

The analogs when $\kappa = 1$ of the following two results were proved in \cite{len_kahler}; the proofs when $\kappa = -1$ proceed along the same lines. The first result establishes several properties of the geometric structure of $K^s$; the second shows that (\ref{pi2HS}) is the geodesic equation on $(K^s, g)$.

\begin{thm}\label{kahlerth}
Let $g$ denote the pseudo-Riemannian metric $\langle \cdot, \cdot \rangle$ on $K^s$. Then the following hold:
\begin{itemize}
\item[(a)] $g$ is a smooth pseudo-Riemannian metric on $K^s$ and $\nabla$ is a smooth connection compatible with $g$.

\item[(b)] $\omega$ is a symplectic form on $K^s$ compatible with $\nabla$, i.e. $\omega$ is a smooth nondegenerate closed two-form on $K^s$ such that $\nabla \omega = 0$.

\item[(c)] $J$ is a smooth $(1,1)$-tensor on $K^s$ such that $J^2 = I$ and $\nabla J = 0$.

\item[(d)] The symplectic form $\omega$, the metric $g$, and the tensor $J$ are compatible in the sense that
$$\omega(U, V) = g(JU, V).$$

\item[(e)] The metric $g$ satisfies 
$$g(U, V) = -g(JU, JV).$$

\item[(f)]  The Nijenhuis-like tensor $N^J$ defined for vector fields $X, Y$ by
$$N^J(X,Y) = [X,Y] - J[JX, Y] - J[X, JY] - [JX, JY]$$
vanishes identically.
\end{itemize}
\end{thm}

\begin{prop}\label{pigeodesicprop}
Let $s > 5/2$. Let $(\varphi, [\alpha]):J \to K^s$ be a $C^2$-curve where $J \subset \R$ is an open interval and define $(u, \rho)$ by (\ref{urhovarphitft}). Then 
\begin{align}\label{piurhoCC1}
& (u, [\rho]) \in C\left([0, T);  H_0^s(S^1) \times (H^{s-1}(S^1)/\R)\right)
	\\ \nonumber
 & \hspace{3cm}\cap C^1\left([0, T); H_0^{s-1}(S^1) \times (H^{s-2}(S^1)/\R)\right)
\end{align}
and $(\varphi, [\alpha])$ is a geodesic if and only if $(u, [\rho])$ satisfies the following weak form of (\ref{pi2HS}):
\begin{align*}
  & \begin{pmatrix} u_t + uu_x \\ \pi(\rho)_t  \end{pmatrix} 
  =  \begin{pmatrix} - \frac{1}{2}A^{-1}\partial_x(u_x^2 - \pi(\rho)^2) \\
-(u\pi(\rho))_x \end{pmatrix}.
\end{align*}
\end{prop}

\begin{rem}
   It is clear from Theorem \ref{kahlerth} that the geometric structure of $K^s$ bears many similarities with a K\"ahler manifold (however, the metric $g$ is not positive definite and $J$ is not a complex structure because $J^2 = I \neq -I$).  
\end{rem}

We next compute the curvature of $K^s$.

\begin{thm}
The curvature tensor $R$ on $K^s$ satisfies
\begin{align}\label{Kcurvature}
\langle R(u, v)v, u\rangle
= \langle u, u \rangle\langle v, v\rangle - \langle u, v\rangle^2 - 3\omega(u,v)^2,
\end{align}
where $u = (u_1, [u_2])$ and $v = (v_1, [v_2])$ are elements in $T_{(\id, [0])}K^s$. In particular, the sectional curvature 
$$\sect(u,v) = \frac{\langle R(u, v)v, u\rangle}{\langle u, u \rangle\langle v, v\rangle - \langle u, v\rangle^2}$$
takes on arbitrarily large positive as well as arbitrarily large negative values.
\end{thm}
\proofbegin
We claim that the natural projection $\projK:G^s \to K^s$ defined by
\begin{align}\label{GtoKprojection}
\projK(\varphi, \alpha) = (\varphi, [\alpha])	
\end{align}
is a semi-Riemannian submersion.\footnote{Recall that a smooth submersion $F$ from $M$ to $N$, where $M$ and $N$ are (possibly weak) pseudo-Riemannian manifolds, is a {\it semi-Riemannian} (or {\it pseudo-Riemannian}) submersion if the restriction of $T_qF$ to the horizontal subspace $(\ker T_qF)^\perp \subset T_qM$ is an isometry onto $T_qN$ for each $q \in M$.
}
Indeed, smoothness of $\projK$ is immediate, and for each $(\varphi, \alpha) \in G^s$, $\projK$ determines the splitting
$$T_{(\varphi, \alpha)}G = (T_{(\varphi, \alpha)}G)^{v} \oplus (T_{(\varphi, \alpha)}G)^{h},$$
where the vertical and horizontal subspaces are defined by
$$(T_{(\varphi, \alpha)}G)^{v} := \ker{T_{(\varphi, \alpha)}\projK} = \{(0, U_2)\; |\; U_2 \text{ is a constant function}\},$$
and
$$(T_{(\varphi, \alpha)}G)^{h}  := (\ker{T_{(\varphi, \alpha)}\projK})^\perp = \{(U_1, U_2)\; |\; \pi(U_2) = U_2\},$$
respectively. The orthogonal projections onto the vertical and horizontal subspaces are given by
$$(U_1, U_2) \mapsto (U_1, U_2)^{v} = \left(0, \int_{S^1} U_2 \varphi_x dx\right)$$
and
\begin{align}\label{horliftU}
  (U_1, U_2) \mapsto (U_1, U_2)^{h} = (U_1, \pi(U_2)),
\end{align}
respectively. Let $U^{h} = (U_1, U_2)$ and $V^{h} = (V_1, V_2)$ be horizontal vectors in $T_{(\varphi, \alpha)}G^s$. Then, since $T\projK(U_1, U_2) = (U_1, [U_2])$, we have
\begin{align*}
\langle U^{h}, V^{h} \rangle_{(\varphi, \alpha)}
&= \frac{1}{4} \int_{S^1} \left(\frac{U_{1x}V_{1x}}{\varphi_x} - \pi(U_2)\pi(V_2)\varphi_x\right)dx
	\\
&= \langle (U_1, [U_2]), (V_1, [V_2]) \rangle_{(\varphi, [\alpha])}
	\\
&=  \langle T\projK(U^{h}), T\projK(V^{h}) \rangle_{(\varphi, [\alpha])},
\end{align*}
showing that $\projK$ is a semi-Riemannian submersion.

O'Neill's formula for semi-Riemannian submersions (see \cite{oneill} p. 213; the formula generalizes to Banach manifolds cf. \cite{L1995} p. 394) implies that
$$\langle R(X,Y)Y, X \rangle_{K^s} = \langle R_{G^s}(\bar{X}, \bar{Y})\bar{Y}, \bar{X} \rangle_{G^s} 
+ \frac{3}{4}\bigl\langle \bigl[\bar{X}, \bar{Y}\bigr]^{v}, \bigl[\bar{X}, \bar{Y}\bigr]^{v} \bigr\rangle_{G^s},$$
where $\bar{X}, \bar{Y}$ denote the horizontal lifts of two vector fields $X, Y$ on $K^s$ and $R_{G^s}$ denotes the curvature tensor on $G^s$.
In view of Corollary \ref{curvcor} and equation (\ref{horliftU}) this yields
\begin{align*}
\langle R(u,v)v, u \rangle_{K^s} & = \langle \bar{u}, \bar{u} \rangle_{G^s} \langle \bar{v}, \bar{v} \rangle_{G^s}
 - \langle \bar{u}, \bar{v}\rangle_{G^s}^2  + \frac{3}{4} \bigl\langle [\bar{u}, \bar{v})]^{v}, [\bar{u}, \bar{v})]^{v}\bigr\rangle_{G^s}
	\\
& =\langle u, u \rangle_{K^s} \langle v, v\rangle_{K^s} - \langle u, v\rangle_{K^s}^2  
+ \frac{3}{4} \bigl\langle [(u_1, \pi(u_2)), (v_1, \pi(v_2))]^{v}, [(u_1, \pi(u_2)), (v_1, \pi(v_2))]^{v}\bigr\rangle_{G^s},
\end{align*}
whenever $u = (u_1, [u_2])$ and $v = (v_1, [v_2])$ are elements of $T_{(\id, [0])}K^s$.
Since 
$$[(u_1, \pi(u_2)), (v_1, \pi(v_2))]^{v} 
= \begin{pmatrix} v_{1x}u_1 - u_{1x}v_1 \\ v_{2x}u_1 - u_{2x}v_1 \end{pmatrix}^{v}
= \begin{pmatrix} 0 \\ \int_{S^1}(v_{2x}u_1 - u_{2x}v_1)dx\end{pmatrix},$$
we find (\ref{Kcurvature}).

Let
$$v = \left( - \frac{\cos(6\pi x)}{6\pi}, \biggl[\frac{1}{2}\sin(2\pi x)\biggr] \right) \in T_{(\id, [0])}K^s$$
and define for every $|a| < 1$ the vector $u_a \in T_{(\id, [0])}K^s$ by
$$u_a = \left(-\frac{\cos(2\pi x)}{2\pi}, \left[\sqrt{1 + a} \sin(4\pi x)\right] \right).$$
A computation shows that
$$\langle u_a, u_a \rangle\langle v, v\rangle - \langle u_a, v\rangle^2 = -\frac{3a}{256}, \qquad
\omega(u_a,v) = \frac{1}{16},$$
and so
$$\sec(u_a,v) = 1 - 3 \frac{\omega(u_a,v)^2}{\langle u_a, u_a \rangle\langle v, v\rangle - \langle u_a, v\rangle^2}
= 1 + \frac{1}{a} \to \pm \infty \quad \text{as} \quad a \to 0^\pm.
$$
This shows that the sectional curvature is unbounded both above and below.
\proofend

\subsection{The quotient space $\mathcal{V}^s$}
In the remainder of this section, we will explore how the geometry of $K^s$ can be understood in terms of the isometry $\Phi$ of Theorem \ref{sphereth}.
We will first show that the pseudosphere $\dS^s$ admits a large group of isometries parametrized by $\beta \in H^s(\S)$. Each isometry $\Lambda_\beta$ in this group is an infinite-dimensional generalization of a Lorentz transformation (or of a hyperbolic rotation) with rapidity for each $x \in \S$ specified by $\beta(x)$.

\begin{prop}\label{Lorentzprop}
For any function $\beta \in H^s(\S)$, the infinite-dimensional Lorentz transformation $\Lambda_\beta:\dS^s \to \dS^s$ defined by
\begin{align}\label{Lorentzdef}
\Lambda_\beta:\begin{pmatrix} f_1 \\ f_2 \end{pmatrix} \mapsto \begin{pmatrix} \cosh{\beta} & -\sinh{\beta} \\
-\sinh{\beta} & \cosh{\beta} \end{pmatrix}
\begin{pmatrix} f_1 \\ f_2 \end{pmatrix}
\end{align}
is a diffeomorphism and an isometry of $\dS^s$ which leaves the subset $\mathcal{U}^s \subset \dS^s$ defined in (\ref{Usdef}) invariant.
\end{prop}
\proofbegin
The invariance of $\mathcal{U}^s$ follows by a straightforward computation, so it is enough to show that $\Lambda_\beta$ viewed as a linear operator on $H^s(\S; \R^2)$, preserves the metric $\llangle \cdot, \cdot \rrangle$. This is easily verified:
\begin{align*}
   \llangle \Lambda_\beta (f_1, f_2) , \Lambda_\beta (g_1, g_2) \rrangle
  = & \;  \int_{\S}\biggl\{ (f_1 \cosh{\beta} - f_2 \sinh{\beta})(g_1 \cosh{\beta} - g_2 \sinh{\beta})
	\\
&\hspace{1cm} - (-f_1 \sinh{\beta} + f_2  \cosh{\beta} )(-g_1 \sinh{\beta} + g_2  \cosh{\beta} )\biggr\} dx
 	\\
= &\;  \int_{\S} (f_1 g_1 - f_2 g_2) dx
= \llangle  (f_1, f_2) , (g_1, g_2) \rrangle.
\end{align*}
\proofend

Assuming that $\beta(x) = \beta \in \R$ is a constant function, Proposition \ref{Lorentzprop} implies that there is a natural action of $\R$ on $\dS^s$ given by 
\begin{align}\label{actiononSs}
  (\beta, f) \mapsto \Lambda_\beta f, \qquad \beta \in \R, \quad f \in \dS^s.
\end{align}  
Under the isometry $\Phi$ of Theorem \ref{sphereth}, this action corresponds to the following action of $\R$ on $G^s$:
\begin{align}\label{actiononGs}
(\beta, (\varphi, \alpha)) \mapsto (\varphi, \alpha - 2\beta), \qquad \beta \in \R, \quad (\varphi, \alpha) \in G^s.
\end{align}
Indeed,
\begin{align*}
\Lambda_\beta \Phi(\varphi, \alpha)
& = \begin{pmatrix} \cosh{\beta} & -\sinh{\beta} \\
-\sinh{\beta} & \cosh{\beta} \end{pmatrix}
\sqrt{\varphi_x} \begin{pmatrix} \cosh(\alpha/2) \\ \sinh(\alpha/2) \end{pmatrix}
= 	\sqrt{\varphi_x}\begin{pmatrix} \cosh(\frac{\alpha}{2} - \beta) \\ \sinh(\frac{\alpha}{2} - \beta) \end{pmatrix}
 = \Phi(\varphi, \alpha - 2\beta).
\end{align*}
The quotient space of $G^s$ under the action (\ref{actiononGs}) is exactly the symplectic manifold $K^s$. Thus, under the isomorphism $\Phi$, $K^s$ corresponds to the quotient space $\mathcal{V}^s$ defined by $\mathcal{V}^s = \mathcal{U}^s/\R$, where two elements $f, \tilde{f} \in \mathcal{U}^s$ are identified iff there exists a $\beta \in \R$ such that $\Lambda_\beta f = \tilde{f}$. The metric $\llangle \cdot, \cdot \rrangle$ on $\dS^s$ induces a metric on $\mathcal{V}^s$ and we have the following result.

\begin{thm}
Define $\Psi:K^s \to \mathcal{V}^{s-1}$ by
$$\Psi(\varphi, [\alpha]) \mapsto \left[\Phi(\varphi, \alpha)\right],$$
and let $\projK:G^s \to K^s$ and $\projV:\mathcal{U}^s \to \mathcal{V}^s$ denote the natural quotient maps. 
Then we have the following commutative diagram:
\begin{align}\label{commutingdiagram}
\xymatrix{\ar[d]_{\projK} G^s \ar[r]^{\Phi} & \mathcal{U}^{s-1} \ar[d]^{\projV} \\
K^s \ar[r]^{\Psi} & \mathcal{V}^{s-1}}
\end{align}
where both $\Phi$ and $\Psi$ are bijective isometries, and the actions of $\R$ on $G^s$ and $\mathcal{U}^{s-1}$ given by (\ref{actiononGs}) and (\ref{actiononSs}) are equivariant with respect to $\Phi$.
\end{thm}

The space $\mathcal{V}^{s}$ has an interesting geometric structure. In Appendix \ref{Bapp} we cast light on this structure by studying the finite-dimensional analog of $\mathcal{V}^{s}$.

\pagebreak

\section{Global weak solutions} \label{main} \nequation
Before constructing weak solutions, we briefly outline a convenient setting for them. 
\subsection{Preliminaries}
Let us define a bilinear operator on $T_{(\id, 0)}G^s \simeq H^s_0(\S) \times H^{s - 1}(\S)$ by 
\begin{equation} \label{gami0}
\Gamma_{(\id,0)} ((u,\rho),(v,\sigma))= \binom{\Gamma^0_{\id} (u,v) +  \frac{1}{2} A^{-1}\partial_x (\rho \sigma)}{-\frac{1}{2} (u_x \sigma + v_x \rho)}, 
\end{equation}
where $\Gamma^0_{\id}(u,v) = -\frac{1}{2} A^{-1} \partial_x (u_x v_x)$ is the Christoffel operator associated with the Hunter-Saxton equation \cite{lenells} and the inverse of $A = -\partial_x^2$ is given by (\ref{Ainv}).
We extend the bilinear operator $\Gamma_{(\id,0)}$ by right invariance to any tangent space $T_{(\varphi, \alpha)}G^s$:
\begin{equation} \label{gam phi}
\Gamma_{(\varphi,\alpha)} (U,V) = \Gamma_{(\id,0)} (U \circ \varphi^{-1}, V \circ \varphi^{-1} ) \circ \varphi, \qquad (\varphi,\alpha) \in G^s, \quad U, V \in T_{(\varphi ,\alpha)}G^s.
\end{equation}
The associated covariant derivative $\nabla$ is defined by
$$
(\nabla_X Y) (\varphi,\alpha) = DY(\varphi,\alpha)\cdot X(\varphi,\alpha) - \Gamma_{(\varphi,\alpha)} (Y(\varphi,\alpha),X(\varphi,\alpha)). 
$$
Finally, by definition, a geodesic in $G^s$ with respect to $\nabla$ is a $C^2$ curve $(\varphi(t),f(t)) \in G^s$ such that
\begin{equation} \label{strong_geod}
(\varphi_{tt}, \alpha_{tt}) = \Gamma_{(\varphi,\alpha)} ((\varphi_t, \alpha_t), (\varphi_t, \alpha_t)). 
\end{equation}

\subsection{Weak geodesic flow}
A weak formulation of the pseudo-Riemannian geodesic equation can be achieved in the framework of the space $\mathcal{M}_{AC} :=M^{AC} \circledS L_2(\S)$ (see \cite{Wun11b}), where $M^{AC}$ is the set of nondecreasing absolutely continuous functions $\varphi: [0,1] \rightarrow [0,1]$ with $\varphi(0) = 0$ and $\varphi(1) = 1$ \cite{len:weak}. 
The tangent space at the identity can be naturally defined (cf. \cite{EE}, page 8) as
$$
T_{(\id,0)} \; \mathcal{M}_{AC} := H_0^1(\S) \times L_2(\S); 
$$
this definition extends by right invariance to the tangent space at any $(\varphi,\alpha)$: 
\begin{equation} \label{rtang}
T_{(\varphi,\alpha)} \; \mathcal{M}_{AC} := \{ (u \circ \varphi,\rho \circ \varphi) : \; (u,\rho) \in T_{(\id,0)}\; \mathcal{M}_{AC} \}. 
\end{equation} 
These tangent spaces can be characterized as follows. 

\begin{lem} \label{character}
Let $AC(\S)$ denote the set of absolutely continuous functions $\mathbb{S} \to \R$. Let $\varphi \in AC(\mathbb{S})$ and write 
\begin{equation} \label{null}
N := \{ x \in \S:\; \varphi_x(x) \mbox{ exists and equals } 0 \}. 
\end{equation}
Then we have the characterization 
\begin{eqnarray*} 
&&T_{ ( \varphi, 0) } \; \mathcal{M}_{AC} = 
\left \{ (U,R) \in AC(\S) \times L_2(\S) :  \; U(0) = 0, \; U_x = 0 \; \mbox{ a.e. on } N,\right. \\
&& \left.  \qquad \qquad \quad \quad \;
 \int_{ \S \setminus N} \frac{U_x^2}{\varphi_x}dx < \infty, \mbox{ and } 
 \int_{\S} R^2 \varphi_x dx < \infty \nonumber
 \right \}. 
\end{eqnarray*} 
Furthermore, for any $(U,R),\;(V,S) \in T_{(\varphi,\alpha)} \mathcal{M}_{AC}$, 
\begin{equation} \label{pseudoMac}
\left\langle (U,R),(V,S) \right\rangle_{(\varphi,\alpha)}= \frac{1}{4} \int_{\S \setminus N} \biggl(\frac{U_x V_x}{\varphi_x} - R\;S\;\varphi_x \biggr)\;dx. 
\end{equation}
\end{lem}
\begin{proof}
The proof of this result follows with straightforward adaptations from \cite{Wun11b}.
\end{proof}
We extend the definition of the Christoffel operator \eqref{gami0} to $\mathcal{M}_{AC}$ by setting, for $\varphi \in M^{AC}$ and $(U,R),\;(V,S) \in T_{(\varphi,\alpha)} \mathcal{M}_{AC}$,
\begin{equation} \label{gamextend}
\Gamma_{(\varphi,\alpha)} ((U,R),(V,S)) = \frac{1}{2} \begin{pmatrix} \int_0^{\varphi(.)} \left( u_x v_x - \rho \sigma\right)\;dx - \varphi(.) \int_{\S} \left( u_x v_x - \rho \sigma \right) \;dx \\ -(u_x \sigma + v_x \rho) \circ \varphi \end{pmatrix}, 
\end{equation}
where $(u,\rho),\;(v,\sigma) \in H_0^1(\S) \times L_2(\S)$ are chosen such that $(U,R) = (u,\rho) \circ \varphi$, $(V,S) = (v,\sigma) \circ \varphi$. 

\noindent 
The following statement asserts the global existence of a geodesic flow on $\mathcal{M}_{AC}$. 

\begin{thm} \label{wgf}
Let $(u_0,\rho_0) \in T_{(\id,0)} \mathcal{M}_{AC} \in H_0^1(\S) \times L_2(\S)$. Let $c = \frac{1}{4} \int_{\S} (u_{0x}^2 - \rho_0^2)dx$ and assume that 
\begin{enumerate}[$(A)$]
\item $c = -1$,
\item $|\rho_0(x)| \le u_{0x}(x) + 2\quad \mbox{ for a.e. }x \in \S$. 
\end{enumerate}  
Define $\varphi(t,x)$ and $\alpha(t,x)$ by
\begin{subequations}\label{varphialphaf1f2}
\begin{align} \label{varphialpha} 
\varphi(t,x) &:= \int_0^x  \left( f_1^2(t,y) - f_2^2(t,y) \right)\;dy, \quad
\alpha(t,x) := 2 \arctanh\frac{f_2(t,x)}{f_1(t,x)} = \rho_0(x) \int_0^t \frac{ds}{\varphi_x(s,x)}, 
\end{align}
where
\begin{align}
f_1(t,x) := \cosh t + \frac{u_{0x}(x)}{2} \sinh t,
\quad f_2(t,x) := \frac{\rho_0(x)}{2} \sinh t, \qquad (t,x) \in [0,\infty) \times \S.
\end{align}
\end{subequations}
Then the following statements are true. 
\begin{enumerate}[$(i)$]
\item For each time $t \ge 0$, $( \varphi(t,.), \alpha(t,.)) \in \mathcal{M}_{AC}$. 
\item For each time $t \ge 0$, $( \varphi_t(t,.), \alpha_t(t,.)) \in T_{(\varphi,\alpha)} \; \mathcal{M}_{AC}$. 
\item 
The geodesic has constant energy for all $t \in [0, \infty)$. More precisely,   
\begin{equation} \label{const:ener}
\left \langle (\varphi_t, \alpha_t), (\varphi_t, \alpha_t) \right \rangle_{(\varphi,\alpha)} = -4, \qquad t \geq 0.
\end{equation}
\item The geodesic equation holds for all $t \in [0,\infty)$: 
\begin{equation} \label{geodesic}
\binom{\varphi_{tt}}{\alpha_{tt}} = \Gamma_{(\varphi,\alpha)} \left( (\varphi_t,\alpha_t), (\varphi_t,\alpha_t) \right), \qquad t \geq 0.
\end{equation}
\end{enumerate}
\end{thm}
\begin{proof}
Assumption $(B)$ implies that the set $N$ defined in (\ref{null}) has measure zero for each $t > 0$. In fact, $(B)$ implies that
\begin{align*}
\varphi_x & = f_1^2 - f_2^2 =  \cosh^2{t} + u_{0x} \cosh{t} \sinh{t} + \frac{u_{0x}^2 - \rho_0^2}{4}\sinh^2{t}
	\\
& \geq \cosh^2{t} - \sinh^2{t} + u_{0x} \sinh{t}(\cosh{t} - \sinh{t}) \quad \text{a.e. on $S^1$},
\end{align*}
and since $u_{0x} \geq -2$ a.e. on $S^1$, this yields
\begin{align}\label{varphixestimate}
  \varphi_x \geq (\cosh{t} - \sinh{t})^2 \quad \text{a.e. on $S^1$}, \quad t \geq 0.
\end{align}  

We can now prove the four statements $(i)$-$(iv)$ in turn.

\smallskip\noindent
{\it Proof of $(i)$.}
This is a result of the definition of $\varphi$, the assumption $c =-1$, and the inequality (\ref{varphixestimate}).

\smallskip\noindent
{\it Proof of $(ii)$.}
In view of Lemma \ref{character} and (\ref{varphixestimate}), it is enough to verify the following conditions: 
\begin{enumerate}[$(1)$]
\item $\varphi_t \in AC(\S)$, $\alpha_t \in L_2(\S)$;
\item $\varphi_t(t,0) = 0$; 
\item $\int_{\S} \alpha_t^2 \varphi_x dx < \infty$. 
\item $\int_{\S} \dfrac{\varphi_{tx}^2}{\varphi_x} dx < \infty$;
\end{enumerate}
Clearly, the map $x \mapsto \varphi_t(t,x)$ is absolutely continuous and $\varphi_t(t,0) = 0$. Moreover, $\alpha_t = \rho_0/\varphi_x \in L_2(\S)$ by \eqref{varphialpha} and (\ref{varphixestimate}).
This proves $(1)$ and $(2)$. The equations \eqref{varphialpha} and (\ref{varphixestimate}) also imply $(3)$. Finally, $(4)$ is a consequence of $(3)$ and equation (\ref{reveal}) below.

\smallskip\noindent
{\it Proof of $(iii)$.}
In view of (\ref{pseudoMac}), we find
\begin{align} \label{reveal} 
\left \langle (\varphi_t, \alpha_t), (\varphi_t, \alpha_t) \right \rangle_{(\varphi,\alpha)}
& = \int_{\S} \left( \frac{\varphi_{tx}(t,x)^2}{\varphi_x(t,x)} - \alpha_{t}^2\;\varphi_x \right) \;dx 
 = 4 \int_{\S} \left( f_{1t}^2 - f_{2t}^2 \right) \;dx
	\\\nonumber
& = 4 \sinh^2 t + \cosh^2 t \int_{\S} ( u_{0x}(x)^2 - \rho_0(x)^2) \;dx 
 \stackrel{(A)}{=} - 4.
\end{align}

\smallskip\noindent
{\it Proof of $(iv)$.}
We have 
$$
\varphi_t(t,x) = 2 \int_0^x (f_1 f_{1t} -  f_2 f_{2t})\;dy, \qquad \varphi_{tt} =  2 \int_0^x (f_{1t}^2 - f_{2t}^2 + f_1 f_{1tt} - f_2 f_{2tt})\;dy.
$$ 
Since $f_{1tt} = f_1$, $f_{2tt} = f_2$, and $\varphi_x = f_1^2 - f_2^2$, we find
\begin{eqnarray} \label{varphitt}
\varphi_{tt}(t,x) = 2 \int_0^x \left( f_{1t}^2 - f_{2t}^2 + f_1^2 - f_2^2 \right) \;dy 
= 2 \int_0^x \left(f_{1t}^2 - f_{2t}^2 \right) dy + 2\varphi(x). 
\end{eqnarray}
From equations \eqref{varphialphaf1f2}, we deduce that 
$$f_{1t}^2 - f_{2t}^2 = \frac{\varphi_{tx}^2 - \rho_0^2}{4 \varphi_x}
= \frac{\varphi_x}{4} \left[(u_x \circ \varphi)^2 - (\rho \circ \varphi)^2\right],$$ 
so that equation (\ref{varphitt}) yields
\begin{align*}
  \varphi_{tt}(t,x) = \frac{1}{2} \int_0^x \varphi_x\left[(u_x \circ \varphi)^2 - (\rho \circ \varphi)^2\right] dy + 2\varphi(x)
= \frac{1}{2} \int_0^{\varphi(x)} \left(u_x^2 - \rho^2\right) dy + 2\varphi(x).
\end{align*}
Since $\int_{\S} (u_x^2 - \rho^2 ) dx = -4$ by (\ref{const:ener}), we find
$$\varphi_{tt} = \Gamma_{(\varphi,\alpha)}^{(1)} \left( (\varphi_t,\alpha_t), (\varphi_t,\alpha_t) \right),  \quad t \in [0, \infty),$$
where $\Gamma^{(1)}$ and $\Gamma^{(2)}$ denote the two components of $\Gamma$.

On the other hand, as $\alpha_t = \rho \circ \varphi$,  one immediately sees that $\alpha_{tt} = - (u_x \rho) \circ \varphi$, and so
$$
\alpha_{tt} = \Gamma_{(\varphi,\alpha)}^{(2)} \left( (\varphi_t,\alpha_t), (\varphi_t,\alpha_t) \right),  \quad t \in [0, \infty).
$$
This finishes the proof of Theorem \ref{wgf}. 
\end{proof}

\noindent
The geodesic formulation \eqref{geodesic} allows us to study weak solutions of the Hunter-Saxton system \eqref{2hs}. 

\begin{defn}
\label{we:so}
The pair $(u,\rho): [0,\infty) \times \S \rightarrow \R^2$ is a global weak solution of equation \eqref{2hs} with initial data $(u_0,\rho_0) \in H^1(\S) \times L_2(\S)$ if 
\begin{enumerate}[$(a)$]
\item for each $t \in [0,\infty)$, the map $x \mapsto u(t,x)$ is in $H^1(\S)$;
\item $u \in C([0,\infty) \times \S; \R)$ and $u(0,.) = u_0$ pointwise on $\S$; $\rho(0,.) = \rho_0\;$ a.e. on $\S$; 
\item the maps $t \mapsto u_x(t,.)$ and $t \mapsto \rho(t,.)$ belong to the space $L^\infty([0,\infty); L_2(\S))$; 
\item the maps $t \mapsto u(t,.)$ and $t \mapsto \rho(t,.)$ are absolutely continuous from $[0,\infty)$ to $L_2(\S)$ and satisfy
\begin{align*}
u_t + u u_x &= \frac{1}{2} \left\{ \int_0^x \left( u_x^2 - \rho^2 \right) \;dy -  x \int_{\S} \left( u_x^2 - \rho^2 \right) \;dy \right\},
\\
\rho_t + (u\rho)_x &= 0
\end{align*}
in $L_2(\S)$ for a.e. $t \in [0,\infty)$. 
\end{enumerate} 
\end{defn}

\noindent
With this definition, we can state the following theorem. 

\begin{thm}
\label{gws}
For any initial data $(u_0,\rho_0) \in H_0^1(\S) \times L_2(\S)$ satisfying the hypotheses of Theorem \ref{wgf}, the pair
\begin{equation} \label{weak:sol}
\binom{u(t,\varphi(t,x))}{\rho(t,\varphi(t,x))} := \binom{\varphi_t(t,x)}{\alpha_t(t,x)}, \quad (t,x) \in [0,\infty) \times \S,  
\end{equation}
constitutes a global weak solution of the Hunter-Saxton system \eqref{2hs} with initial data $(u_0,\rho_0)$. 
Moreover, this solution is conservative:  
$$
\int_{S^1} (u_x^2 - \rho^2) dx = -4, \quad  t  \in [0,\infty).  
$$ 
\end{thm}
\begin{proof}
The proof of this theorem follows, mutatis mutandis, the lines of the proof of Theorem 4.2 of \cite{Wun11b}; see also \cite{len:weak}. 
\end{proof}

\begin{rem}
In contrast with the Riemannian-metric case ($\kappa = 1$), the solutions in Theorem \ref{gws} are only periodic in space, and not in time as well (cf. {\rm Figure \ref{fig:a=2}} and \cite{Wun11b,len_kahler}).  
\end{rem}

\begin{rem}
Several problems remain unsolved: For example, the construction of global weak ``spacelike'' and ``lightlike'' geodesics and the corresponding weak solutions of \eqref{2hs}, i.e., solutions with initial data $(u_0, \rho_0)$ satisfying $(i)$ $c > 0$ (spacelike) or $(ii)$ $c = 0$ (lightlike), remains open.
One obstruction here is that there do not seem to be any reasonable assumptions for the initial data ensuring that the geodesics avoid hitting the boundary. Obviously, the requirement $u_x (x) > |\rho(x)|$ for all $x$ in $(ii)$ -- as used in \cite{GY2011} if $x \in \R$ -- cannot be carried over to the periodic case. Finally, it would be desirable to relax the condition $(B)$ in Theorem \ref{wgf}. 
\end{rem}

\appendix
\section{The curvature of $G^s$} \label{Aapp}
\renewcommand{\theequation}{A.\arabic{equation}}\nequation
In this appendix, we give a direct proof that the curvature of $G^s$ is constant and equal to $1$. 

The Arnold formula for the curvature of a Lie group with a right-invariant metric is
\begin{align}\label{arnoldcurvformula}
  \langle R(u,v)v, u \rangle = \langle \delta, \delta \rangle + \langle [u,v], \beta \rangle - \frac{3}{4}\langle [u,v], [u,v]\rangle - \langle B(u,u), B(v,v) \rangle,
\end{align}
where 
$$\delta := \frac{1}{2} \left(B(u,v) + B(v,u) \right), \qquad \beta := \frac{1}{2} \left(B(u,v) - B(v,u) \right),$$
the bilinear map $B$ is defined by
$$\langle B(u,v), w \rangle = \langle u, [v, w]\rangle,$$
and $u,v,w$ are tangent vectors at the identity. In the case of $G^s$, we have 
$$\langle u, v \rangle = \frac{1}{4}\int_{\S} (u_{1x}v_{1x} + \kappa u_2v_2) dx \qquad \text{and} \qquad
 [u,v] = \begin{pmatrix} v_{1x}u_1 - u_{1x}v_1 \\ v_{2x}u_1 - u_{2x}v_1 \end{pmatrix},$$
where $u = (u_1, u_2)$ and $v = (v_1, v_2)$ are elements of $T_{(\id, 0)} G^s$. Thus,
$$B(u,v) = \begin{pmatrix} A^{-1}(u_{1xx}v_{1x} + (u_{1xx}v_1)_x - \kappa u_2v_{2x}) \\
- (u_2v_1)_x \end{pmatrix},$$
which implies that
\begin{align*}
\delta & = 
\frac{1}{2} \begin{pmatrix} A^{-1}\partial_x\left[u_{1xx} v_1 + v_{1xx} u_1 + v_{1x} u_{1x} - \kappa u_2 v_{2}\right]  \\
 -(u_2 v_1 + v_2 u_1)_x
  \end{pmatrix},
	\\
\beta & = 
\frac{1}{2} \begin{pmatrix} A^{-1}\left[(u_{1xx} v_1)_x + v_{1x} u_{1xx} - \kappa u_2 v_{2x}
- (v_{1xx} u_1)_x - u_{1x} v_{1xx} + \kappa v_2 u_{2x}\right]  \\
 (-u_2 v_1 + v_2 u_1)_x
  \end{pmatrix}.
\end{align*}

Using the identity
\begin{align*}
  & -\partial_x A^{-1} \partial_x	f = f - \int_{\S} f dx,	
\end{align*}
we can compute the four terms in (\ref{arnoldcurvformula}). The first term is given by
\begin{align*}
\langle \delta, \delta \rangle
= &
-\frac{1}{16}\int_{\S}  \left[u_{1xx} v_1 + v_{1xx} u_1 + v_{1x} u_{1x} - \kappa u_2 v_{2}\right]
	\\
&\qquad\qquad \times \partial_xA^{-1}\partial_x\left[u_{1xx} v_1 + v_{1xx} u_1 + v_{1x} u_{1x} - \kappa u_2 v_{2}\right] dx
+ \frac{\kappa}{16}\int_{\S} (u_2 v_1 + v_2 u_1)_x^2 dx
  	\\
= &\;
\frac{1}{16}\int_{\S}  \left[u_{1xx} v_1 + v_{1xx} u_1 + v_{1x} u_{1x} - \kappa u_2 v_{2}\right]^2 dx
	\\
&- \frac{1}{16}\left(\int_{\S}  \left[u_{1xx} v_1 + v_{1xx} u_1 + v_{1x} u_{1x} - \kappa u_2 v_{2}\right] dx \right)^2
 + \frac{\kappa}{16}\int_{\S} (u_2 v_1 + v_2 u_1)_x^2 dx.
\end{align*}	
The second term is given by
\begin{align*}
\langle [u,v], \beta \rangle
=&\; \frac{1}{8} \int_{\S} (v_{1x}u_1 - u_{1x} v_1)
\bigl[(u_{1xx} v_1)_x + v_{1x} u_{1xx} - \kappa u_2 v_{2x}
 - (v_{1xx} u_1)_x - u_{1x} v_{1xx} + \kappa v_2 u_{2x}\bigr]dx 
	\\
& + 
\frac{\kappa}{8}  \int_{\S} (v_{2x}u_1 - u_{2x} v_1) (-u_2 v_1 + v_2 u_1)_x dx. 
\end{align*}	
The third term is given by
\begin{align*}
- \frac{3}{4}\langle [u,v], [u,v]\rangle 
= - \frac{3}{16}\int_{\S} (v_{1x}u_1 - u_{1x}v_1)_x^2 dx
- \frac{3\kappa}{16}\int_{\S} (v_{2x}u_1 - u_{2x}v_1)^2 dx.
\end{align*}
The fourth term is given by
\begin{align*}
 - \langle B(u,u), B(v,v) \rangle
  = &\; \frac{1}{4} \int_{\S} \left[u_{1xx} u_1 + \frac{1}{2}u_{1x}^2 -  \frac{\kappa}{2}u_2^2\right] \partial_xA^{-1}\partial_x\left[v_{1xx}v_1 + \frac{1}{2}v_{1x}^2 -  \frac{\kappa}{2}v_2^2\right]dx
 	\\
& - \frac{\kappa}{4}\int_{\S} (u_2 u_1)_x(v_2 v_1)_x dx
	\\
 = & - \frac{1}{4} \int_{\S} \left[u_{1xx} u_1 + \frac{1}{2}u_{1x}^2 -  \frac{\kappa}{2}u_2^2\right] \left[v_{1xx}v_1 + \frac{1}{2}v_{1x}^2 - \frac{\kappa}{2}v_2^2\right]dx
	\\
&  + \frac{1}{4}\left(\int_{\S} \left[u_{1xx} u_1 + \frac{1}{2}u_{1x}^2 -  \frac{\kappa}{2}u_2^2\right]dx\right)
\left(\int_{\S} \left[v_{1xx}v_1 + \frac{1}{2}v_{1x}^2 -  \frac{\kappa}{2}v_2^2\right]dx\right)
	\\
& - \frac{\kappa}{4}\int_{\S} (u_2 u_1)_x(v_2 v_1)_x dx.
\end{align*}
Summing up the above four contributions, we infer that the terms in $\langle R(u,v)v, u\rangle$ that do {\it not} contain $u_2$ or $v_2$ are given by
\begin{align*}
&\frac{1}{16}\int_{\S}  \left[u_{1xx} v_1 + v_{1xx} u_1 + v_{1x} u_{1x}\right]^2dx
- \frac{1}{16}\left(\int_{\S}  \left[u_{1xx} v_1 + v_{1xx} u_1 + v_{1x} u_{1x} \right] dx \right)^2
 	\\	
& + \frac{1}{8} \int_{\S} (v_{1x}u_1 - u_{1x} v_1)
\bigl[(u_{1xx} v_1)_x + v_{1x} u_{1xx} - (v_{1xx} u_1)_x - u_{1x} v_{1xx} \bigr]dx 
	\\
& - \frac{3}{16}\int_{\S} (v_{1x}u_1 - u_{1x}v_1)_x^2dx
- \frac{1}{4} \int_{\S} \left[u_{1xx} u_1 + \frac{1}{2}u_{1x}^2\right] \left[v_{1xx}v_1 + \frac{1}{2}v_{1x}^2\right]dx
	\\
& +
\frac{1}{4} \left(\int_{\S} \left[u_{1xx} u_1 + \frac{1}{2}u_{1x}^2 \right]dx\right)
\left(\int_{\S} \left[v_{1xx}v_1 + \frac{1}{2}v_{1x}^2\right]dx\right)
	\\
= &\; \frac{1}{16}\int_{\S}  \left[u_{1xx} v_1 + v_{1xx} u_1 + v_{1x} u_{1x}\right]^2 dx
- \frac{1}{16}\left(\int_{\S}  v_{1x} u_{1x} dx \right)^2 
 	\\	
& + \frac{1}{8} \int_{\S} (v_{1xx}u_1 - u_{1xx} v_1)^2dx 
 + \frac{1}{8} \int_{\S} (v_{1x}u_1 - u_{1x} v_1)
\bigl[v_{1x} u_{1xx}  - u_{1x} v_{1xx} \bigr]dx 
	\\
& - \frac{3}{16}\int_{\S} (v_{1x}u_1 - u_{1x}v_1)_x^2 dx
- \frac{1}{4} \int_{\S} \left[u_{1xx} u_1 + \frac{1}{2}u_{1x}^2\right] \left[v_{1xx}v_1 + \frac{1}{2}v_{1x}^2\right]dx
	\\
& +
\frac{1}{16} \left(\int_{\S} u_{1x}^2 dx\right) \left(\int_{\S} v_{1x}^2 dx\right)
	\\
=&\; \frac{1}{16} \left(\int_{\S} u_{1x}^2 dx\right)
\left(\int_{\S} v_{1x}^2 dx\right) - \frac{1}{16}\left(\int_{\S}  v_{1x} u_{1x} dx \right)^2.	
\end{align*}
On the other hand, the terms in $\langle R(u,v)v, u\rangle$ that contain $u_2$ or $v_2$ are given by
\begin{align*}
& \frac{1}{16}\int_{\S}  \left[-2\kappa \left(u_{1xx} v_1 + v_{1xx} u_1 + v_{1x} u_{1x}\right)u_2 v_{2} + (u_2 v_{2})^2\right]dx
	\\
&- \frac{1}{16}\left(-2\kappa\left(\int_{\S}  \left[u_{1xx} v_1 + v_{1xx} u_1 + v_{1x} u_{1x}\right] dx \right)\left(\int_{\S} u_2 v_{2} dx\right)
+ \left(\int_{\S} u_2 v_{2} dx\right)^2\right)
	\\
& + \frac{\kappa}{16}\int_{\S} (u_2 v_1 + v_2 u_1)_x^2 dx
+ \frac{\kappa}{8} \int_{\S} (v_{1x}u_1 - u_{1x} v_1)\bigl[-u_2 v_{2x} + v_2 u_{2x}\bigr]dx 
	\\
& + 
\frac{\kappa}{8}  \int_{\S} (v_{2x}u_1 - u_{2x} v_1) (-u_2 v_1 + v_2 u_1)_x dx
 - \frac{3\kappa}{16} \int_{\S} (v_{2x}u_1 - u_{2x}v_1)^2  dx
	\\
& + \frac{\kappa}{8} \int_{\S} \left[u_{1xx} u_1 + \frac{1}{2}u_{1x}^2 \right] v_2^2dx
 + \frac{\kappa}{8} \int_{\S} u_2^2\left[v_{1xx}v_1 + \frac{1}{2}v_{1x}^2\right]dx -  \frac{1}{16}\int_{\S} u_2^2v_2^2dx
	\\
& -
\frac{\kappa}{8} \left(\int_{\S} \left[u_{1xx} u_1 + \frac{1}{2}u_{1x}^2\right]dx\right)
\left(\int_{\S} v_2^2dx\right)
-
\frac{\kappa}{8} \left(\int_{\S} u_2^2dx\right)
\left(\int_{\S}\left[v_{1xx}v_1 + \frac{1}{2}v_{1x}^2\right]dx\right)
	\\
&+
\frac{1}{16} \left(\int_{\S} u_2^2 dx\right)\left(\int_{\S} v_2^2 dx\right)
-
\frac{\kappa}{4} \int_{\S} (u_2 u_1)_x(v_2 v_1)_x dx		
	\\
= &\; \frac{1}{16} \left(\int_{\S} u_2^2 dx\right)\left(\int_{\S} v_2^2 dx\right)
- \frac{1}{16}  \left(\int_{\S} u_2 v_{2}dx \right)^2
	\\
& - \frac{\kappa}{8}\left(\int_{\S} u_{1x}v_{1x} dx \right)\left(\int_{\S} u_2 v_2 dx \right)
+ \frac{\kappa}{16}\left(\int_{\S} u_{1x}^2dx\right)\left(\int_{\S} v_2^2 dx \right)
+ \frac{\kappa}{16}\left(\int_{\S} u_2^2 dx\right)\left(\int_{\S} v_{1x}^2 dx \right)
\end{align*}
In summary, we arrive at
\begin{align*}
\langle R(u,v)v, u\rangle
= &\; \frac{1}{16} \left(\int_{\S} u_{1x}^2 dx + \kappa \int_{\S} u_2^2 dx \right)\left(\int_{\S} v_{1x}^2 dx + \kappa \int_{\S} v_2^2 dx\right)
	\\
& - \frac{1}{16}\left(\int_{\S} u_{1x}v_{1x} dx + \kappa \int_{\S} u_2v_2 dx\right)^2
	\\
= &\; \langle u, u\rangle\langle v, v\rangle - \langle u, v\rangle^2.
\end{align*}	
This shows that $G^s$ has constant sectional curvature $1$ when $\kappa = -1$.

\begin{rem}
The above proof is valid also when $\kappa = 1$.
\end{rem}

\section{The finite-dimensional analog of $\mathcal{V}^s$} \label{Bapp}
\renewcommand{\theequation}{B.\arabic{equation}}\nequation
In the case of $\kappa = 1$, the space $\mathcal{V}^s$ introduced in Section \ref{geometrysec} is an infinite-dimensional analog of complex projective space $\mathbb{C}P^n$, see \cite{len_kahler}. In the case of $\kappa = -1$, the finite-dimensional analog of $\mathcal{V}^s$ is an interesting pseudo-Riemannian manifold which will be explored in this appendix.  

Let $n \geq 1$. Define the submanifold $\mathcal{S}$ of $\R^{2(n+1)}$ by
$$\mathcal{S} = \left\{ x = \begin{pmatrix} x_1 & \cdots & x_{n+1} \\ y_1 & \cdots & y_{n+1} \end{pmatrix} \in \R^{2(n+1)} \middle| 
\sum_{i=1}^{n+1} (x_i^2 - y_i^2) = 1 \right\},$$
and equip $\mathcal{S}$ with the (indefinite) metric $g$ induced by the bilinear form
$$ds^2 = \sum_{i=1}^{n+1} (dx_i^2 - dy_i^2).$$
Using the coordinates $(y_1, \dots, y_{n+1}, x_1, \dots x_{n+1})$ in the definition (\ref{pseudospheredef}), we see that $\mathcal{S}$ is nothing but the pseudosphere $S_{n+1}^{2n+1}(r)$ with radius $r = 1$.
The real numbers act by isometries on $\mathcal{S}$ by 
$$\beta \cdot x = \Lambda_\beta x \quad \text{where} \quad \Lambda_\beta = \begin{pmatrix} \cosh{\beta} & -\sinh{\beta} \\
-\sinh{\beta} & \cosh{\beta} \end{pmatrix}, \quad \beta \in \R.$$
This action is smooth, free, and proper, so the orbit space $\mathcal{S}/\R$ admits a unique smooth manifold structure such that the quotient map $q:\mathcal{S} \to \mathcal{S}/\R$ is a smooth submersion. We endow $\mathcal{S}/\R$ with the induced metric so that $q$ becomes a semi-Riemannian submersion.

\begin{rem}
The finite-dimensional analog of the subset $\mathcal{U}^s$ defined in (\ref{Usdef}) is the subset $\mathcal{U} \subset \mathcal{S}$ given by 
$$\mathcal{U} = \left\{x \in \mathcal{S} | x_i > 0 \text{ and } x_i^2 - y_i^2 > 0  \text{ for }  1 \leq i \leq n+1\right\}.$$
The finite-dimensional analog of $\mathcal{V}^s$ is the subset $\mathcal{U}/\R \subset \mathcal{S}/\R$.
\end{rem}

\noindent
The vertical distribution in $T\mathcal{S}$ is spanned by the timelike vector field $V$ whose value at $x$ is given by
$$V = \frac{d}{d\beta}\biggl|_{\beta = 0} \Lambda_\beta x = Jx,
\quad \text{where} \quad J = \begin{pmatrix} 0 & -1 \\ -1 & 0 \end{pmatrix}.$$
Clearly, $J^2 = I$ and 
\begin{align*}
g(V, V) = -g(x, x) = -1,\qquad 
X^v = -g(X, V) V, \qquad
X^h = X + g(X, V) V,
\end{align*}
where $X^v$ and $X^h$ denote the vertical and horizontal components of the tangent vector
$$X = \begin{pmatrix} X_1 & \cdots & X_{n+1} \\ Y_1 & \cdots & Y_{n+1} \end{pmatrix} \in T_{x}\mathcal{S}.$$ 
Moreover,
\begin{align*}
& g(JX, JX') = -g(X, X'), \qquad
g(JX, X) 
= 0, \qquad X, X' \in T_{x} \mathcal{S}.
\end{align*}
The horizontal subspace at $x \in \mathcal{S}$ is a $2n$-dimensional subspace with $n$ timelike and $n$ spacelike dimensions given by
$$(T_x\mathcal{S})^h
= \left\{X \in \R^{2(n+1)} \middle| g(X, x) = 0 \text{ and } g(X, V) = 0\right\}.$$
In particular, $J$ leaves the horizontal distribution invariant. 
Since $J\Lambda_\beta = \Lambda_\beta J$, $J$ descends to $\mathcal{S}/\R$.
The two-form $\omega$ defined by
$$\omega(X, X') = g(JX, X'),$$
also descends to $\mathcal{S}/\R$.

Let $\bar{X}, \bar{Y}$ be the horizontal lifts of two vector fields $X, Y$ on $\mathcal{S}/\R$. Then, letting $\bar{\nabla}$ denote the covariant derivative in the ambient space $(\R^{2(n+1)}, ds^2)$, we find (cf. \cite{P2006} p. 86)
$$g\Bigl(\frac{1}{2}[\bar{X}, \bar{Y}], V\Bigr) = g(\nabla_{\bar{X}}^{\mathcal{S}}\bar{Y}, V)
= g(\bar{\nabla}_{\bar{X}} \bar{Y}, V)
= -g(\bar{Y}, \bar{\nabla}_{\bar{X}}V)
= g(\bar{Y}, J \bar{\nabla}_{\bar{X}} JV)
= g(\bar{Y}, J \bar{X}).$$
Thus,
$$\frac{1}{2}[\bar{X}, \bar{Y}]^v = -g(\bar{Y}, J\bar{X}) V.$$
Since the sectional curvature of $\mathcal{S}$ is constant and equal to $1$, O'Neill's formula yields the following finite-dimensional analog of equation (\ref{Kcurvature}):
\begin{align} \nonumber
g(R_{\mathcal{S}/\R}(X, Y)Y, X) 
& = g(R_{\mathcal{S}}(\bar{X}, \bar{Y})\bar{Y}, \bar{X}) + \frac{3}{4} g([\bar{X}, \bar{Y}]^v , [\bar{X}, \bar{Y}]^v)
	\\ \label{finitedimcurv}
& = g(X, X)g(Y,Y) - g(X,Y)^2 -  3\omega(X,Y)^2.
\end{align}

\noindent 
Any subspace of $T(\mathcal{S}/\R)$ spanned by two vectors of the form $X, JX$ has sectional curvature $4$. Indeed, equation (\ref{finitedimcurv}) implies
$$\sec_{\mathcal{S}/\R}(X, JX) =  1 - 3\frac{ g(JX, JX)^2}{g(X,X)g(JX,JX) - g(X, JX)^2} = 4.$$
It follows that $\mathcal{S}/\R$ has constant curvature equal to $4$ when $n = 1$. However, for $n \geq 2$, it is easy to see that the curvature of $\mathcal{S}/\R$ is not constant (consider for example the different subspaces spanned by pairs of basis elements of $(T_x\mathcal{S})^h$ where $x = (x_1, \dots, x_{n+1}, y_{1}, \dots, y_{n+1}) = (1, 0, \dots, 0)$). Proposition 28 on p. 229 of \cite{oneill} then implies that the curvature when $n \geq 2$ takes on arbitrarily large positive as well as arbitrarily large negative values.

\section*{Acknowledgments}
JL acknowledges support from the EPSRC, UK. \\MW acknowledges support from the ETH Foundation.

 \end{document}